\documentclass[10pt,a4paper]{amsart}
\usepackage{hyperref}
\usepackage{mathrsfs}
\usepackage{amsmath}
\usepackage{amssymb}
\usepackage[all]{xy}
\usepackage{latexsym}
\usepackage[utf8]{inputenc}
\title[Indecomposable $\aone$-modules]{Truncated projective spaces, Brown-Gitler 
spectra and indecomposable 
$\aone$-modules}
\author[Geoffrey Powell]{Geoffrey Powell}
\address{Laboratoire Angevin de Recherche en Mathématiques, UMR 6093, 
  Faculté des Sciences, Université d'Angers, 
2 Boulevard Lavoisier,
49045 Angers, France}
\email{Geoffrey.Powell@math.cnrs.fr}
\keywords{Steenrod algebra, indecomposable module, truncated projective space, 
Brown-Gitler spectra}
\subjclass[2000]{19L41; 55S10}
\date{}
\thanks{The author is very grateful to an anonymous referee for remarks which have lead to an improvement in the presentation of the paper.}

\newtheorem{thm}{Theorem}[section]
\newtheorem{prop}[thm]{Proposition}
\newtheorem{cor}[thm]{Corollary}
\newtheorem{lem}[thm]{Lemma}

\theoremstyle{definition}
\newtheorem{defn}[thm]{Definition}
\newtheorem{exam}[thm]{Example}

\theoremstyle{remark}
\newtheorem{rem}[thm]{Remark}
\newtheorem{nota}[thm]{Notation}
\newtheorem{conv}[thm]{Convention}
\newcommand{\obstruct}{\omega}

\newcommand{\rp}{\mathbb{R}P}
\newcommand{\calc}{\mathscr{C}}
\newcommand{\cale}{\mathscr{E}}
\newcommand{\orb}{\mathscr{O}}
\newcommand{\e}{\mathscr{K}}
\newcommand{\stext}{\mathscr{E}xt_{\aone}}
\newcommand{\projhom}{\mathrm{ProjHom}_{\aone}}
\newcommand{\wfour}{\mathscr{T}_0}
\newcommand{\wtwo}{\mathscr{T}}
\newcommand{\wt}{\mathrm{wt}}
\newcommand{\red}{^{\mathrm{red}}}
\newcommand{\pica}{\mathrm{Pic}_{\aone}}
\newcommand{\aone}{{\cala (1)}}
\newcommand{\azero}{{\cala (0)}}
\newcommand{\image}{\mathrm{image}}
\newcommand{\eone}{{E(1)}}
\newcommand{\ext}{\mathrm{Ext}}
\newcommand{\nat}{\mathbb{N}}
\newcommand{\zed}{\mathbb{Z}}
\newcommand{\field}{\mathbb{F}}
\newcommand{\cala}{\mathscr{A}}
\renewcommand{\hom}{\mathrm{Hom}}
\renewcommand{\phi}{\varphi}
\renewcommand{\epsilon}{\varepsilon}
\begin{document}

\begin{abstract}
A structure theorem for bounded-below modules over the subalgebra $\aone$ of 
the 
mod $2$ Steenrod 
algebra generated by $Sq^1, Sq^2$ is proved; this is applied to prove a 
classification theorem  for a family of indecomposable $\aone$-modules. 
The action of the $\aone$-Picard group on this family is described, as is the 
behaviour of duality.

The cohomology of dual Brown-Gitler spectra is identified within this family and the 
relation with members of the
$\aone$-Picard group is made explicit. Similarly, the cohomology of truncated 
projective spaces is considered within this classification; this leads to 
a conceptual understanding of various results within the literature. In 
particular, a unified approach to $\ext$-groups relevant to Adams spectral 
sequence calculations is obtained, 
englobing earlier results of Davis (for truncated projective spaces) and recent 
work of Pearson (for Brown-Gitler spectrum). 
\end{abstract}
\maketitle

\section{Introduction}

The study of modules over $\aone$, the finite sub-Hopf algebra of the mod $2$ 
Steenrod algebra generated by $Sq^1$ and $Sq^2$, has topological 
significance through the use of the Adams spectral sequence to calculate 
cohomology or homology for connective orthogonal theory, $ko$: 
a change of rings reduces to calculating the $E_2$-term as $\ext_{\aone}$ in 
the category of $\aone$-modules. This is a classical topic, with 
significant applications; for example, calculations using truncated projective 
spaces have applications to non-immersion results, occurring in work
of Davis, Mahowald and many others (see \cite{DGM,DGMerr}, for example).

The aim of this paper is to provide a unified approach to a number of 
related questions, for example elucidating the  relationship 
between  truncated projective spaces and  dual Brown-Gitler 
spectra, as seen through the eyes of $ko$ (or, more prosaically, 
 representations of $\aone$). These and related modules are of significant 
interest; for example, 
Mahowald's theory of $bo$-resolutions \cite{mah_bo,mah_bo_add}  
relies  upon an understanding of such modules.

From the algebraic viewpoint, the cornerstone of this work is provided by the 
results of Section \ref{sect:basic}:  they contain the essence of a number of
classification results, 
including the calculation of the Picard group $\pica$, due to Adams and Priddy
\cite{adams_priddy}, and the classification of the local Picard groups, due to 
Yu \cite{yu}. The main algebraic result is Theorem \ref{thm:fundamental}, which 
is applied in Section \ref{sect:applications} to obtain new proofs of the above 
results. For the purposes of this introduction, the algebraic result can be 
paraphrased informally as follows, 
using the Margolis cohomology groups $H^* (M, Q_j)$, $j \in \{0,1\}$ of an 
$\aone$-module $M$: 
if $M$ is bounded below and has an isolated lowest dimensional Margolis 
cohomology class, then there 
exists an inclusion of a small member of the Picard group $\pica$ into $M$ 
which 
carries the lowest dimensional 
Margolis cohomology class. 

Using these techniques,  in Section \ref{sect:finite_indec} 
a  classification is given of the  reduced finite $\aone$-modules such that the
underlying module over $E(1) =  
\Lambda (Q_0, Q_1)$  is indecomposable (up to free factors) - see Theorem 
\ref{thm:classification}. This of significant interest, since the cohomology of
truncated 
projective spaces and of Brown-Gitler spectra  and their duals (up to free
modules) fit into this
family. 
 
Whereas the situation for $\eone$ is easy,  
there is currently no full classification of the finite, 
indecomposable $\aone$-modules: the hypothesis that the module remains 
indecomposable (up to 
free factors)
after restriction to $E(1)$ is far from anodyne. For example, tensor products 
of certain members of the above family  throw up an infinite family of 
non-trivial examples where indecomposability is 
not preserved. 

In Section \ref{sect:picard_interpret} it is shown that, up to action 
 of the Picard group, there are natural choices of representatives for 
 stable isomorphism classes in the above family: the algebraic results of the 
paper
 give a natural choice of orbit representatives under the Picard group. It 
turns 
out
 that existing results in the literature can be understood more conceptually 
using these representatives. 

 Sections \ref{sect:trunc_p} and \ref{sect:brown-gitler} are devoted 
to the analysis of the cohomology of truncated projective spaces 
 and of  dual Brown-Gitler spectra. One of the key conclusions is the
identification 
of the cases where the cohomology of a truncated projective space is
 stably equivalent as an $\aone$-module to that of a dual Brown-Gitler spectrum
(up 
to suspension). 
 A further point of significant interest is the relationship between the family 
of finite indecomposable modules considered here, 
elements of the Picard group and the cohomology of infinite real projective
space.
 
The theory is applied in Section \ref{sect:ext} to obtain (see Theorem
\ref{thm:calculate_stext_Ake}) the  
$\ext_\aone$-groups for the modules covered by the above classification theorem.
For simplicity, the calculations are  presented in terms of stable ext, 
$\stext$,  defined in terms of the stable category of $\aone$-modules; the usual
$\ext$-groups can easily be deduced, the only additional work  required being
for $\hom_{\aone}$, since 
$\stext^0$ does not see the morphisms factoring across projectives. This
recovers calculations by Davis \cite{davis} and, for dual 
Brown-Gitler spectra, by Pearson \cite{pearson}. The method used here is
simplified significantly
 by using the presentation of the modules in terms of elements of the
$\aone$-Picard group, $\pica$; a further leitmotif
 is that all calculations should be carried out graded over $\pica$.

The above methods  are  algebraic: truncated projective spaces 
and dual Brown-Gitler spectra only occur via their cohomology with
$\field_2$-coefficients. 
Section \ref{sect:bgp} sketches the relationship between Brown-Gitler spectra 
and classifying spaces of elementary abelian $2$-groups as seen in the stable
$\aone$-module category. 
This illustrates the utility of the  above $\stext$ calculations. These  
are also applied in Section \ref{sect:tensor} to consider the decomposition of
tensor products of 
modules of the form appearing in Theorem \ref{thm:classification}. This  has
been addressed for truncated projective spaces by Davis in 
\cite[Theorem 3.9]{davis}; studying the question in the light of the
classification of families obtained in Theorem 
\ref{thm:classification} and the subsequent cohomological calculations renders
the situation more transparent.

The original motivation for this work came from a desire to understand the
relationship between the 
results of the author's paper \cite{powell}, which considered the $ko\langle n 
\rangle$-cohomology of elementary abelian $2$-groups, 
 and Pearson's work \cite{pearson} on the $ko$-homology of Brown-Gitler
spectra. 
This is explained in  Section \ref{sect:toda}; in particular, this gives a
second order approximation (as 
compared to the results of \cite{powell}) to a detection property which is 
implicitly established by Cowen-Morton \cite{cowen_morton} for the Hopf modules
(for 
mod-$2$ homology) of the theories $ko \langle n \rangle$ over the Hopf ring of 
$ko$, using the relationship between homology of Brown-Gitler spectra and Hopf 
rings  \cite{goerss}.

\section{Algebraic background}

The underlying prime is taken to be $2$ and $\field$ 
denotes the 
prime field $\field_2$;
$\aone$ is the sub Hopf algebra of the Steenrod algebra generated by $Sq^1,
Sq^2$ and $\eone$  that generated by $Q_0  = Sq^1$ and $Q_1= [Sq^2, Sq^1]$.
Thus $\eone$ is the exterior, 
primitively-generated Hopf algebra $\Lambda (Q_0, Q_1) $ and there is a short
exact sequence of Hopf algebras:
\[
 \field 
\rightarrow 
\eone 
\rightarrow 
\aone
\rightarrow
\Lambda (\overline{Sq^2}) 
\rightarrow 
\field,
\]
where $\overline{Sq^2}$ denotes the image of $Sq^2$.
 
An 
$\aone$-module is reduced if and only if $Sq^2
Sq^2 Sq^2$ acts trivially and an $\eone$-module is reduced if and only if $Q_0
Q_1$ acts trivially. A bounded-below $\aone$-module $M$ can be written as 
$M \cong F \oplus M\red$, where $F$ is a free $\aone$-module and $M\red$ is 
reduced (see \cite[Proposition 
2.1]{bruner_Ossa}).

\begin{nota}
 Write $\simeq$ to denote stable isomorphism of $\aone$-modules.
\end{nota}

An essential ingredient is the Adams and Margolis \cite{adams_margolis} 
criterion for a morphism of $\aone$-modules to be a stable isomorphism:
namely, a morphism  $f : M\rightarrow N$ between bounded-below $\aone$-modules 
is a stable
isomorphism if and only if it induces an isomorphism  $H^* (f, Q_j)$ on 
Margolis 
cohomology groups, for $j \in \{ 0, 1 \}$.

The stable $\aone$-module category has for objects $\aone$-modules and 
morphisms $[M,N]:= \hom_\aone(M, N) / \projhom (M,N)$ where $\projhom (M,N)$ is 
the space 
of  morphisms which factor through a projective module. For further details on
the category of $\aone$-modules and
the associated stable module categories, the reader is referred to \cite[Section
2]{bruner_Ossa}.

The following observation is useful:

\begin{lem}
\label{lem:stable_hom_reduced}
For $M$ a reduced $\aone$-module, the quotient map 
$
 \hom_{\aone} (\field, M) \rightarrow 
[\field , M]
$
 is an isomorphism. 
\end{lem}

\begin{nota}
 The duality functor on the category of $\aone$- (respectively $\eone$-) 
modules 
(induced by vector space duality) is denoted by $D$.
\end{nota}

\begin{rem}
\label{rem:left_right_A1}
The conjugation $\chi$ of $\aone$ acts by $\chi (Sq^i) = Sq^i$, for 
$i \in \{1, 2 \}$,  since the characteristic is two,  $Sq^1$ is 
primitive, and $(Sq^1)^2=0$. Hence, a distinction between right and left 
modules over $\aone$ is unnecessary and, in considering the dual of a module 
represented by a 
diagram giving the action of $Sq^1$ and $Sq^2$, it suffices to reverse the 
direction of the arrows and interpret the degrees correctly. 
\end{rem}

The structure of the category of $\eone$-modules is well understood. For
instance, in \cite[Theorem III.16.11]{adams}, Adams gives a classification of
the finite-dimensional
 indecomposable $\eone$-modules,  given by the free modules of rank one 
together with the family of finite lightning-flash modules (recall that a 
lightning-flash module is a sub-quotient of a bi-infinite module as represented 
in Figure \ref{fig:lightning}). 
Moreover, he shows that every finite $\eone$-module is a finite direct sum of 
modules of this type.

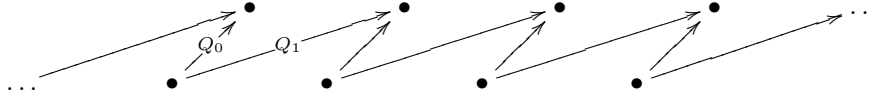
\begin{figure}
 \caption{Representation of a  lightning-flash $E(1)$-module}
\label{fig:lightning}
\[
 \xymatrix @C=1.5pc @R=1.5pc{
&
&&
\bullet 
&&
\bullet 
&&
\bullet 
&&
\bullet 
&&
\ldots
\\
\ldots
\ar[urrr]
&&
\bullet
\ar[ur]|{Q_0}
\ar[urrr]|{Q_1}
&&
\bullet
\ar[ur]
\ar[urrr]
&&
\bullet
\ar[ur]
\ar[urrr]
&&
\bullet
\ar[ur]
\ar[urrr]
&&
}
\]
\end{figure}

There is no analogous  classification of indecomposable $\aone$-modules in the
literature. An 
important point is that a
reduced $\aone$-module is not in general reduced when restricted to $\eone$; 
indeed, for $M$ an $\aone$-module, $M|_{\eone}$ is reduced if and
only if $Sq^2Sq^2$ acts trivially on $M$.

\begin{nota}
Write $\Omega \field$ for the augmentation ideal of $\aone$  and $\Omega^{-1} 
\field$ for its dual (which can be identified 
with  the coaugmentation ideal, with appropriate degree shift) and,   
for $M$ an $\aone$-module and $n \in \zed$, 
\[
 \Omega ^n M := \left\{ 
\begin{array}{ll}
  \big( (\Omega \field) ^{\otimes n} \otimes M \big)\red & n \geq 0
\\
  \big( (\Omega^{-1} \field) ^{\otimes |n|} \otimes M \big)\red & n < 0.
\end{array}
\right.
\]
\end{nota}

\begin{rem}
The functor $\Omega \field \otimes - $ on $\aone$-modules induces an equivalence
of the stable module category  with inverse induced by $\Omega^{-1} \field 
\otimes - $, since $\Omega \field \otimes \Omega^{-1} \field $ is stably 
isomorphic to $\field$.
(Note  that $\Omega^n \Omega^{-n} M $ is 
non-canonically  isomorphic to $M\red$.)
\end{rem}

\begin{lem}
\label{lem:Omega_Margolis}
For $M$ an $\aone$-module, there are natural isomorphisms
\begin{eqnarray*}
 H^* (M,Q_0) & \cong & H^{*+1} (\Omega M , Q_0) \\
H^* (M, Q_1)& \cong & H^{*+3} (\Omega M, Q_1).
\end{eqnarray*}
\end{lem}

\begin{proof}
 An application of the long exact sequences for $H^* (-,Q_j)$, $j \in \{0,1 \}$.
\end{proof}

Adams and Priddy \cite{adams_priddy}
calculated the structure of the Picard group,  $\pica$, namely 
the group of stable isomorphism classes of stably  invertible $\aone$-modules 
with respect to $\otimes$:
\[
 \pica \cong \zed ^{\oplus 2} \oplus \zed /2,
\]
 generated by $\Sigma \field$, $\Omega \field$ and $J$, the Joker, 
the unique element of $\pica$ of order two (that is $J^{\otimes 2} \simeq 
\field$). The 
Margolis 
cohomology groups of $J$ are necessarily one-dimensional concentrated in
degree $0$, 
$
 H^0 (J, Q_0 ) = H^0 (J, Q_1) =\field,
$
hence do not distinguish between $J$ and $\field$. However, there is no 
non-trivial morphism between the two, so they are not stably isomorphic; 
indeed, 
the 
structure of $J$ is represented by the diagram
\smallskip

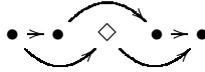
\[
\xymatrix @C=.5pc @R=.5pc{
\bullet 
\ar@/_1pc/[rr]
\ar[r]
&
\bullet 
\ar@/^1pc/[rr]
&
\Diamond 
\ar@/_1pc/[rr]
&
\bullet
\ar[r]
&
\bullet
}
\]
\ \smallskip

\noindent
in which the  generator $\Diamond$ represents the non-trivial Margolis 
cohomology 
classes.

\begin{rem}
 Elements of the Picard group arise in the work of Milgram \cite{milgram}; his 
results were
used by Davis, Gitler and Mahowald \cite{DGM}, where the notation $Q_{j,n}$ was 
introduced. This obscures the group structure of $\pica$ (cf. \cite[Lemma 
3.11]{DGM}),  hence  is not adopted here. 
\end{rem}

\begin{exam}
\label{exam:Q1_picard_family}
 The modules $Q_{1,j}$, for $j \in \nat$, can be identified as follows:
\[
 Q_{1,j} \simeq
\left\{ 
\begin{array}{ll}
 (\Sigma^{-1} \Omega) ^{4l+1} J & j =2l;\\
 (\Sigma^{-1} \Omega) ^{4l+3} \field & j = 2l+1 .
\end{array}
\right.
\]
A module from this family is uniquely determined (up to stable isomorphism) by 
its Margolis cohomology groups. (These modules are illustrated, along with the 
other families $Q_{i,j}$, in \cite[Figure 1]{bailey}; the reader 
should compare these diagrams with \cite[Figure 1]{bruner_Ossa}.)

The dual modules identify  as:
\[
 D Q_{1,j} \simeq
\left\{ 
\begin{array}{ll}
(\Sigma \Omega^{-1}) ^{4l+1} J & j =2l;\\
 (\Sigma \Omega^{-1}) ^{4l+3} \field  & j = 2l+1 .
\end{array}
\right.
\]
Note that these do not fit into the same families under the action of 
$(\Sigma^{-1} \Omega) ^{4l}$, $l \in \zed$.
\end{exam}

\begin{figure}
 \caption{Diagrammatic representations of certain 
$\aone$-modules}
\label{fig:diag_aone}
\[
\xymatrix @C=.5pc @R=.3pc{
& J && \Omega^{-1}  J && \Omega^{-2} J && \field && \Omega^{-1} \field \\
2 & \bullet &&&&&&&\\
1 & \bullet \ar[u] \\
0 & \bullet \ar@/_1pc/[uu]&&&&&& \bullet \\
-1 & \bullet \ar@/^1pc/[uu] && \bullet &&&&&& \bullet &  \\
-2 & \bullet \ar@/_1pc/[uu] \ar[u] & &&& \bullet &&&&&\bullet \\
-3 & &&\bullet \ar@/_1pc/[uu] &&&&&& \bullet \ar@/^1pc/[uu] &\bullet \ar[u]\\
-4 & &&\bullet \ar[u] && \bullet \ar@/_1pc/[uu]&&&&&\bullet \ar[lu] 
\ar@/^1pc/[uu]\\
-5 & & & &&\bullet \ar[u]&&&& & \bullet \ar@/_1pc/[uu]\\
-6 & & & &&\bullet &&&&&\bullet .\ar[u] \ar@/^1pc/[uu] \\
-7 & & & &&\bullet \ar@/_1pc/[uu] \ar[u]
}
\]
\end{figure}
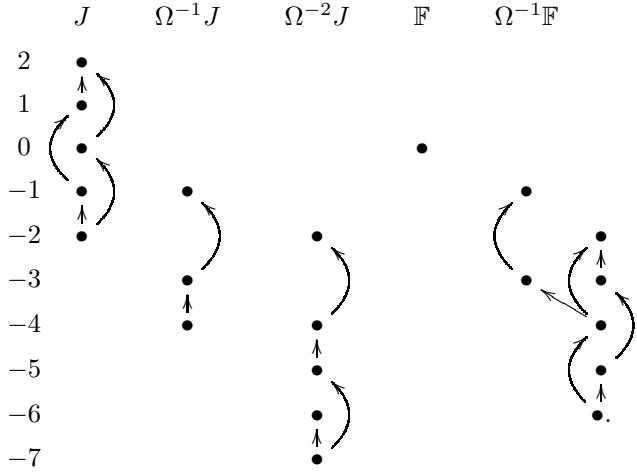

\begin{rem}
 The module $\Omega^{-1} J
$ illustrated in Figure \ref{fig:diag_aone} is the  question-mark complex 
(with the appropriate degrees) and
$\Omega^{-1} \field$ is the quotient of $\Sigma^{-6}\aone$ by its socle, 
$\field$. 
\end{rem}

\begin{rem}
\label{rem:A1A0}
The $\aone$-module $\aone \otimes _{\azero} \field$ is also of importance; it
is 
represented by the following diagram:

\begin{eqnarray}
\label{eqn:A1A0}
\xymatrix @C=1pc @R=.5pc{
\bullet 
\ar@/^1pc/[rr]
&&
\bullet 
\ar[r]
&
\bullet
\ar@/^1pc/[rr]
&&
\bullet.
}
\end{eqnarray}
\end{rem}

\section{Low dimensional behaviour of $\aone$-modules}
\label{sect:basic}

This section addresses the fundamental question: given a bounded-below 
$\aone$-module $M$ and knowledge of the structure of 
$(M|_{E(1)})\red$, what can be said about the structure of $M$, in particular 
in low dimensions. The information on  $(M|_{E(1)})\red$
 is provided by the Margolis cohomology groups.

Propositions \ref{prop:technical} and \ref{prop:tech2} have a
technical appearance  but are the foundations for many key structural results 
for $\aone$-modules.
More-user-friendly statements are given in Theorem \ref{thm:fundamental} and 
Corollary \ref{cor:embedding}. First applications of these results are 
given in Section \ref{sect:applications} and the main application in Section 
\ref{sect:finite_indec}.

\begin{prop}
\label{prop:technical}
 Let $M$ be a reduced $\aone$-module and  $d \in \zed$ such that $M = M^{\geq 
d}$ 
and $M^d \neq 0$. 
\begin{enumerate}
 \item 
If $H^d (M, Q_0)=0 = H^{d+1} (M, Q_1)$, then  
\begin{enumerate}
\item 
$H^d (M, Q_1) =0$;
\item
$Sq^2 Sq^2 : M^d \rightarrow M^{d+4}$ is injective, in particular $Sq^2 : M^d
\rightarrow M^{d+2}$ is injective; 
\item 
the operation $Q_0 = Sq^1$ on $M^{d+2}$ induces a short exact sequence
\[
 0
\rightarrow 
V 
\rightarrow 
\image (Sq^2) ^{d+2} 
\rightarrow 
W
\rightarrow 
0,
\]
where   $V :=  (\image (Sq^2) \cap \ker (Q^0) )  ^{d+2} $ and $W := \image (Sq^1
Sq^2 : M^d \rightarrow M^{d+3}  )$;
\item
$W \subset \ker (Q_0) \cap \ker (Q_1)$ and the natural map $W \rightarrow
H^{d+3} (M , Q_1) $ is injective; 
\item
$V \subset \ker (Q_0) \cap \ker(Q_1)$ and 
\begin{enumerate}
\item
the natural map $V \rightarrow H^{d+2} (M, Q_1)$ is injective; 
\item 
the kernel of the natural map $V \rightarrow H^{d+2} (M, Q_0)$ lies in the image
of $Q_0 : M^{d+1} \rightarrow M^{d+2}$;
\end{enumerate}
\item
setting $\tilde{V} = (\image (Sq^2) \cap \ker (Sq^2 Sq^1) )  ^{d+2} $, 
$V \subset \tilde{V} \subset \ker (Q_1)$ and the natural map $\tilde{V} 
\rightarrow H^{d+2}
(M, Q_1)$ is injective; 
\end{enumerate}
in particular,  $H^{d+2} (M, Q_1) \oplus H^{d+3} (M, Q_1) $ is non-trivial. 
\item 
\label{item:d+2}
If $H^d (M, Q_0)=0 = H^{d+1} (M, Q_1)$ (as above) and $H^{d+2} (M, Q_1) =0$,
then 
\begin{enumerate}
\item 
$Sq^1 Sq^2$ induces an injection  $ M^d \hookrightarrow H^{d+3}(M, Q_1)$;
\item 
the operation $Sq^2 Sq^1 Sq^2 : M^d \rightarrow M^{d+5}$ is injective. 
\end{enumerate}
\end{enumerate}
\end{prop}
 
\begin{proof}
Suppose that $M = M^{\geq d}$ is reduced, $M^d \neq 0$  and $H^d (M, Q_0)=0=
H^{d+1} (M, Q_1)$;  the $Q_0$ hypothesis implies that $Q_0 : M^d \rightarrow
M^{d+1}$ is injective (since $M^{<d}=0$) and, similarly, 
$Q_1 : M^{d+1} \rightarrow M^{d+4}$ is injective, by the $Q_1$ hypothesis. It
follows that $Q_0 Q_1 = Q_1 Q_0 : M^d \rightarrow M^{d+4}$ is injective;
this shows 
that $H^d (M, Q_1) =0$ and $Sq^2 Sq^2 = Q_0 Q_1$ acts injectively on $M_d$. 
The definitions of $V, W$ in terms of the operation of $Q_0$ on the image of
$Sq^2 : M^d \rightarrow M^{d+2}$ provide the short exact sequence.
 
By definition, $V$ lies in the kernel of  $Q_0$ and hence in $\tilde{V}$; 
moreover $\tilde{V}\subset  \ker (Q_1)$,
since $Q_1 Sq^2 = Sq^2 Sq^1 Sq^2$. For connectivity reasons, the image of $Q_1$ 
is trivial in this degree,  hence $\tilde{V}$ injects to $H^{d+2} (M, Q_1)$. 

Restricting to $V$,  the kernel of the map $V\rightarrow H^{d+2} (M , Q_0)$ 
lies, by definition of Margolis homology,  in the
image $Q_0 : M^{d+1} \rightarrow M^{d+2}$.

Now consider $W$; it is clear that $W$ lies in the kernel of $Q_0$ and the
hypothesis that $M$ is reduced implies that it also lies in the
kernel of $Q_1$. To see that $W \cap \image (Q_1)$ is trivial, use the fact
established above that $Q_0 Q_1$ acts injectively on $M^d$, hence the
image of $Q_1 : M^d \hookrightarrow M^{d+3}$ has trivial intersection with $\ker
(Q_0)$. This provides the inclusion $W \subset H^{d+3} (M, Q_1)$. 

The statements of part (\ref{item:d+2}) are  consequences of the
above; if $H^{d+2}(M, Q_1)=0$, then $V=0$, which is equivalent to the 
injectivity of $Sq^1 Sq^2: M^d \rightarrow M^{d+3}$ 
and the injection corresponds to $W \hookrightarrow H^{d+3}(M, Q_1)$. 
Similarly, $\tilde{V} \subset H^{d+2} (M, Q_1) =0$ implies that $Sq^2 Sq^1 
Sq^2$ 
acts injectively on $M^d$. 
\end{proof}

\begin{cor}
\label{cor:nonvanishing_Margolis}
 Let $M$ be a reduced $\aone$-module and $d \in \zed$ such that $M = M^{\geq 
d}$, $M^d \neq 0$,  
 and $H^d(M, Q_0) = 0$,  then $H^{d+t} 
(M,
Q_1) \neq 0$ for some $ t \in \{1, 2, 3 \}$.
\end{cor}

\begin{nota}
For $M$ a graded vector space and $n \in \zed$, as usual $M^n$ denotes the 
component of degree $n$; when forming the tensor product with a graded module 
$[M^n]$ will be written to emphasize that this is considered as a vector space 
placed in degree zero. 
\end{nota}

The reader may wish to recall the diagrammatic representation of $\aone 
\otimes_{\azero} \field$ (see Remark \ref{rem:A1A0}) in relation to the 
following.

\begin{prop}
\label{prop:tech2}
 Let $M$ be a reduced $\aone$-module  and $d \in \zed$ such 
that $M = M^{\geq d}$,  $H^d(M, Q_0) \neq 0$ and $H^d (M, Q_1)=0$.

Then
\begin{enumerate}
 \item 
$H^d(M; Q_0) \cong \ker \{M^d \stackrel{Q_0}{\rightarrow}  M^{d+1}\} \neq 0$;
\item 
$Sq^1 Sq^2 : H^d(M; Q_0) \rightarrow M^{d+3}$ is injective;
\item 
if $H^{d+2} (M, Q_1) = 0$, then $Sq^2Sq^1 Sq ^2 : H^d(M; Q_0)\rightarrow 
M^{d+5}$ is injective, 
hence induces a monomorphism
\[
 [H^d(M; Q_0)] \otimes  \Sigma^d  (\aone \otimes_{\cala(0)} \field ) 
\hookrightarrow M.
\]
\end{enumerate}
\end{prop}

\begin{proof}
 The reasoning is similar to that employed in the proof of Proposition
\ref{prop:technical}.
\end{proof}

For notational simplicity, the following is stated for $d=0$ (in the notation 
of 
the previous results).

\begin{thm}
\label{thm:fundamental}
 Let $M$ be a reduced $\aone$-module which is bounded-below and such that 
  $H^{<0} (M, Q_j)=0$ for  $j \in \{0,1 \}$ (so that $M = M ^{\geq -3}$).
\begin{enumerate}
 \item 
If $M^{-3} \neq 0$, then there is a canonical inclusion:
\[
[ M^{-3}] \otimes \Sigma^3 \Omega^{-1} \field \hookrightarrow M,
\]
the image $Sq^1 Sq^2(M^{-3}) $  lies in $\ker (Q_1)$
and the composite
\[
 M^{-3}\stackrel{Sq^1 Sq^2}{ \rightarrow}M^0 \rightarrow H^0 (M, Q_1) 
\]
is a monomorphism,  in particular is non-trivial.
\item
If $M^{-3} =0 $, $M^{-2} \neq 0$, set $V^{-2} := \ker(Sq^1 Sq^2)^{-2}$ and 
$W^{1}:= \image (Sq^1 Sq^2) ^1$; then 
 there is a canonical inclusion 
\[
[ V^{-2} ] \otimes  J \rightarrow M, 
\]
the image $Sq^2 (V^{-2})$ lies in $\ker (Q_1)$ 
and the composite 
\[
 V^{-2} \stackrel{Sq^2}{\rightarrow} M^0 \rightarrow H^0 (M, Q_1) 
\]
is a monomorphism.

Similarly, there is an injection $W^1 \hookrightarrow H^1 (M,Q_1)$; in 
particular, if $H^1 (M, Q_1)=0$, then $V^{-2} = M^{-2}$.
\item
If $M^{-3}=0= M^{-2}$ and $M^{-1} \neq 0$, set $X := \ker (Sq^2) ^{-1} $ and $Y
:= \ker (Sq^2Sq ^2) ^{-1}$ so that $X \subset Y \subset M^{-1}$;
\begin{enumerate}
\item 
$Q_0 Y$ lies in $\ker (Q_1)$ and yields a  subspace $Q_0 Y \subset H^0 (M,
Q_1)$;
\item
 $Q_1$ acts injectively on $X$ by $Sq^2 Sq^1$,  hence induces a monomorphism
\[
[ X ] \otimes \Sigma ^3\Omega^{-1} J \hookrightarrow M
\]
\item 
if $X=0$ and $H^1 (M, Q_1) =0$, then $Sq^2Sq ^1 Sq^2  : M^{-1} \rightarrow
M^4$ is injective. 
If, in addition,  $H^2 (M, Q_1)=0$ then
$Y= M^{-1}$ and there is a monomorphism 
\[
 [M^{-1}] \otimes \Sigma^6\Omega^{-2} J \hookrightarrow M. 
\]
This restricts to an inclusion 
$
 [M^{-1}] \otimes  \field \hookrightarrow M. 
$
\end{enumerate}
\end{enumerate}
\end{thm}

\begin{proof}
 The result follows from Proposition \ref{prop:technical}.

In the case $M^{-3}\neq 0$, $H^{-3} (M, Q_0) =0 = H^{-2}(M, Q_1) = H^{-1} (M,
Q_1)$ hence  $V=0$ (in the notation of the Proposition), both $Sq^2 Sq^2$ and 
$Sq^2 Sq^1
Sq^2$ act injectively on 
$M^{-3}$, and $Sq^1 Sq^2$ embeds $M^{-3}$ into $H^0(M, Q_1)$. This corresponds 
to
the stated embedding 
of copies of $\Sigma^3 \Omega^{-1}\field$ into $M$.

 In the case $M^{-3}= 0$ and $M^{-2} \neq 0$, with  $H^{-2} (M, Q_0) =0 = 
H^{-1}(M,
Q_1)$ by hypothesis,  $Sq^2 Sq^2 $ acts injectively on $M^{-2}$; by 
definition,  
 $Sq^1 Sq^2$ acts trivially on $V^{-2} $ and $Sq^2$  embeds $V^{-2}$
into $H^0 (M, Q_1)$. This corresponds to the embedding of copies of $J$ 
into $M$.
The embedding of $W^1$ into $H^1 (M, Q_1)$ is given in  Proposition 
\ref{prop:technical}; in particular, if 
 $H^1(M, Q_1) =0$, then  $W^1 =0$ and $V^{-2}=M^{-2}$.

 The remaining case $M^{-3}= 0 = M^{-2}$ and $M^{-1} \neq 0$ is slightly more
delicate. The statement concerning $Q_0 Y$ is clear, 
from the relation $Sq^2 Sq^2 = Sq^1 Sq^2 Sq^1$, and the embedding corresponding
to $X$ is obtained as above. 

If $X=0$, then $Sq^2 : M^{-1} \rightarrow M^1$ is injective; $Sq^2 (M^{-1}) \cap
\ker (Sq^2 Sq^1)$ lies in the kernel of $Q_1$ hence embeds into $H^1 (M, Q_1)$. 
If the latter is trivial, it follows that $Sq^2 Sq^1 Sq^2$ acts injectively on 
$M^{-1}$. 

The argument is concluded by using the hypothesis $H^2 (M, Q_1)=0$.  By the 
above, $Sq^2$ acts injectively on 
$Sq^1 Sq^2( M^{-1})$,  whereas $Sq^2$ acts trivially upon the image of $Sq^2 
Sq^1$, so that 
$Sq^1 Sq^2 (M^{-1}) \cap Sq^2 Sq^1 (M^{-1}) = 0$ and hence 
$
 Sq^1 Sq^2 (M^{-1}) \oplus Sq^2 Sq^1 (M^{-1}) \subset \big(\ker (Q_1)\big) ^{2}.
$
Consider the map 
$$Q_1 : M^{-1} \rightarrow Sq^1 Sq^2 (M^{-1}) \oplus Sq^2 Sq^1 (M^{-1}) \subset 
\big(\ker (Q_1)\big)  ^{2}.
$$
The component to $Sq^1 Sq^2 (M^{-1})$ is injective, hence $Sq^2 Sq^1 (M^{-1}) 
\cap Q_1 (M^{-1}) =0$. 
Thus, if $H^2 (M, Q_1)=0$, $Sq^2 Sq^1 (M^{-1}) = 0$. This implies that 
$Sq^2  Sq^2 (M^{-1}) = Sq^1 Sq^2 Sq^1 (M^{-1}) =0$, so that $Y=M^{-1}$ and, 
moreover, 
$Sq^1 (M^{-1})$ lies in the socle of $M$. The conclusion follows. 
\end{proof}

\begin{rem}
The modules which appear in the above statement, 
$\Omega^{-1} \field$, $J$, $\Omega^{-1} J$, $\Omega^{-2} J$ 
(up to suspensions) all lie in $\pica$. (See Figure \ref{fig:diag_aone} for 
a schematic representation of these modules.)
\end{rem}

The conclusion of Theorem \ref{thm:fundamental} is made more concrete (in a 
special case) in Corollary \ref{cor:embedding} below. Recall that the module 
$\Sigma^3 \Omega ^{-1}  J$ is 
the question-mark complex, suspended so that the $Q_1$-Margolis cohomology is 
in 
degree zero; the following generalizes 
Lemma \ref{lem:stable_hom_reduced}.

\begin{lem}
 \label{lem:question_stable}
For $M$ a reduced $\aone$-module, the quotient 
$ 
 \hom_{\aone} (\Sigma^3 \Omega ^{-1}  J, M) 
\twoheadrightarrow 
[\Sigma^3 \Omega ^{-1}  J, M]
$ 
is an isomorphism.
\end{lem}

\begin{proof}
 This follows from the fact that each non-trivial quotient of $\Sigma^3 
\Omega ^{-1}  J$ has simple socle.
\end{proof}

\begin{cor}
\label{cor:embedding}
 Let $M$ be a bounded-below $\aone$-module such that $H^{\leq 0} (M, Q_0)= 0= 
H^{<0} (M, Q_1) $, $H^1 (M, Q_1) =0$  and $H^0 (M, Q_1)\neq 0$. Then, for 
at least  one of  
$N \in \{ M \red , (J \otimes M)\red \}$, the following statements hold: 
\begin{enumerate}
 \item 
There exists a monomorphism of one of the following forms:
\begin{enumerate}
\item
$\field \hookrightarrow N$
\item 
$\Sigma^3 \Omega^{-1} J \hookrightarrow N$
\end{enumerate}
which induces an injection on $H^* (-, Q_1)$.  

In particular, one of the  following is non-trivial:
\[
 [\field, M],\  [\field, M \otimes J],\  [\Sigma^3 \Omega^{-1} J , M ], \ 
[\Sigma^3 \Omega^{-1} J , M\otimes J  ] = [\Sigma^3 \Omega^{-1} \field , M  ].
\]
\item 
If $H^0 (M, Q_1) = \field$ then $N = N^{\geq -1}$. 
\end{enumerate}
\end{cor}

\begin{proof}
The hypotheses rule out the possibility that $M = M^{\geq 0}$. For suppose 
$M=M^{\geq 0}$, then  $Sq^2 Sq^2=Q_0Q_1$ does not act injectively on $M^0$, 
since $H^0(M,Q_1)\neq 0$. Now  $Q_0:\big ( \ker (Sq^2 Sq^2)\big) ^0 \rightarrow 
M^1$ is injective, since $H^0 (M , Q_0)=0$ and induces an injection to $H^1 (M, 
Q_1)$. By hypothesis $H^1 (M,Q_1)=0$, hence 
this provides a contradiction. 

Theorem \ref{thm:fundamental} now can be applied in each of  the cases $M= 
M^{\geq -t}$, $t \in \{ 1, 2, 3\}$, giving a morphism of the required form;  
the 
injectivity of the morphism is a simple consequence of the arguments used in 
the 
proof of Theorem \ref{thm:fundamental}.

 For instance, an injection $\Sigma^3 \Omega^{-1} \field \hookrightarrow M\red$ 
which is an 
isomorphism on $H^0 (-, Q_1)$ induces a map
\[
 \Sigma^3 \Omega^{-1} J \rightarrow (M \otimes J)\red
\]
 which is an isomorphism on $H^0 (-, Q_1)$. It is straightforward to check that 
this is necessarily injective. 

The conclusion on the non-triviality of one of the stable homomorphism spaces 
follows from Lemmas \ref{lem:stable_hom_reduced} and \ref{lem:question_stable}.

Finally, the connectivity statement follows by using the hypothesis $H^0(M,Q_1) 
= 
\field$ to exclude additional classes in degrees $-3$ and $-2$.
\end{proof}

\begin{rem}
The hypothesis that $H^1 (M, Q_1)=0$ in Corollary \ref{cor:embedding} is 
necessary so as to apply Theorem \ref{thm:fundamental}. There are analogous 
statements without this hypothesis, but at the price of elegance. This is 
addressed in Section \ref{sect:finite_indec}, notably in the proof of Theorem 
\ref{thm:classification}.
\end{rem}

When the lowest 
Margolis cohomology group 
is $Q_0$-cohomology, Theorem \ref{thm:fundamental} gives:

\begin{cor}
\label{cor:connectivity_Q1_acyclic}
For $M$ a reduced, bounded-below $\aone$-module, if  $$H^{< 0} (M, Q_0)= 
0= 
H^{\leq 1} (M, Q_1), $$
then  $M = M^{\geq 0}$ and $M^0 \neq 0 $ if and only if $H^0 (M, Q_0) \neq 0$.
\end{cor}

\begin{proof}
 Theorem \ref{thm:fundamental} shows that the vanishing of $H^0 (M, Q_1)$ and 
$H^1 (M, Q_1)$ implies that $M^{-3} = M^{-2} = M^{-1} =0$. The final point 
follows as in Corollary \ref{cor:embedding}.
\end{proof}

\section{First applications}
\label{sect:applications}

Theorem \ref{thm:fundamental} can be seen as the key ingredient in the proof of 
a number of results, as explained below, and 
is applied in Section \ref{sect:finite_indec} to prove a new classification 
result.

\begin{exam}
 Theorem \ref{thm:fundamental} includes the main step used by Adams-Priddy
\cite{adams_priddy} in their calculation 
of the Picard group, $\pica$. Namely, as in \cite{adams_priddy}, one reduces to 
the case 
where the Margolis cohomology groups 
are concentrated in degree zero (and are both one-dimensional). Then, if the 
module is not stably isomorphic to $\field$, Theorem 
\ref{thm:fundamental}  provides an embedding of $J$, which induces 
an isomorphism of Margolis cohomology groups, hence is a stable isomorphism.
\end{exam}

\begin{nota}
\cite{bruner_Ossa}
 Let $P_0$ denote the unique (up to isomorphism) reduced $\aone$-module
satisfying the following properties:
\begin{enumerate}
 \item 
$P_0$ is bounded-below;
\item 
$P_0$ is $Q_0$-acyclic and $H^* (P_0, Q_1)$ is one-dimensional, concentrated in
degree zero;
\item 
there is an inclusion $\field \hookrightarrow P_0$ which induces an isomorphism
on $H^* (-, Q_1)$. 
\end{enumerate}
\end{nota}

\begin{rem}
 The module $P_0$ is realized topologically by the mod $2$ cohomology of the 
Thom spectrum $\rp^\infty_{-1}$. 
\end{rem}

From the above characterization it follows that $P_0$ is stably idempotent 
($P_0 
\otimes P_0$ is stably isomorphic to $P_0$) 
and, more generally, if $M$ is a bounded-below $\aone$-module which is 
$Q_0$-acyclic, then $\field \rightarrow P_0$ induces a stable 
isomorphism 
\[
 M \stackrel{\simeq}{\rightarrow} P_0 \otimes M.
\]

\begin{prop}
\label{prop:P0_extension}
 Let $M$ be a bounded-below $\aone$-module which is $Q_0$-acyclic and $N$  a
bounded-below $\aone$-module equipped with an $\aone$-linear morphism $f: N 
\rightarrow M$ 
which induces an isomorphism $H^* (f, Q_1)$ in $Q_1$-Margolis cohomology. Then
$f$ induces a stable isomorphism between $P_0 \otimes N $ and $M$.
\end{prop}

\begin{proof}
By the Adams and Margolis \cite{adams_margolis} criterion, it suffices to show 
that the morphism induces an isomorphism on 
$H^* (-, Q_j)$ for $j \in \{ 0, 1 \}$. Applying the
Künneth isomorphism for Margolis cohomology, it follows that the following are
stable isomorphisms
\[
 \xymatrix{
P_0 \otimes N \ar[r]_{\mathrm{Id}\otimes f}^\simeq 
&
P_0 \otimes M
& M .
\ar[l]_(.4)\simeq 
}
\]
The result follows. 
\end{proof}

As a consequence, one obtains a proof of a result first proven (by an 
intricate  calculational method) in Yu's thesis \cite{yu}; a  much simpler 
proof  is  given in \cite{bruner_Ossa}.

\begin{thm}
\label{thm:yu}
 Let $P$ be a bounded-below $\aone$-module which satisfies the following:
\begin{enumerate}
 \item 
$P$ is $Q_0$-acyclic;
\item 
$H^* (P, Q_1)$ is one-dimensional, concentrated in degree zero. 
\end{enumerate}
Then $P$ is stably isomorphic to one of the following:
\[
 P_0, \ P_0 \otimes \Sigma^3 \Omega^{-1} J, \ P_0 \otimes J, \  P_0 \otimes
\Sigma ^3 \Omega^{-1} \field.
\]
Moreover, the inclusion $\field \rightarrow \Sigma^6 \Omega^{-2} J $ induces a
stable isomorphism 
\[
 \Omega^2 P_0 \stackrel{\simeq}{\rightarrow} \Sigma^6 J \otimes P_0,
\]
hence, the above are stably isomorphic to 
\[
 P_0, \ \Sigma^{-3} \Omega P_0, \ \Sigma^{-6} \Omega^2 P_0 \simeq J \otimes P_0 
,
\  \Sigma^{-9} \Omega^3 P_0 \simeq \Sigma^{-3} \Omega J \otimes P_0.
\]
\end{thm}

\begin{proof}
Apply  Proposition \ref{prop:P0_extension} to the 
morphisms
provided by Theorem \ref{thm:fundamental}. 
\end{proof}

\begin{rem}
\label{rem:top_real_Pn}
\ 
\begin{enumerate}
 \item 
 In the notation of \cite[Section 4]{bruner_Ossa}, these modules are 
$
 P_0$, $\Sigma^{-2}P_1$,   $\Sigma^{-4} P_2$, $ \Sigma^{-6} P_3$.
See \cite[Figure 1]{bruner_Ossa} for diagrammatic representations.
\item 
In Bruner's notation, the module $P_1$ is realized as the reduced mod 
$2$-cohomology of $\rp^\infty$ and, for $n \in \{1, 2, 3\}$, there is an 
isomorphism   of $\aone$-modules $(P_1 ^{\otimes n})\red \cong P_n$; thus, up
to 
stable isomorphism of $\aone$-modules, $P_n$ is realized as the reduced mod 
$2$ cohomology of  $(\rp^\infty)^{\wedge n}$.
\end{enumerate}

\end{rem}

For later use, the following standard result is recalled:

\begin{prop}
 \label{prop:Q_0-acyclic}
Let $M$ be a bounded-below $\aone$-module which is $Q_0$-acyclic, then there 
are 
stable isomorphisms:
\begin{eqnarray*} 
\Omega^4 M & \simeq & \Sigma^{12} M 
\\
\Omega^2 M & \simeq & \Sigma^6 J \otimes M.
\end{eqnarray*}
\end{prop}

\begin{proof}
 These properties hold for $P_0$, by Theorem \ref{thm:yu}, whence the result 
via 
the stable isomorphism $P_0 \otimes M \simeq M$.
\end{proof}

\section{A Margolis-type killing construction for $Q_1$-cohomology}
\label{sect:killing}

Motivated by Corollary \ref{cor:embedding} and for use in the proof of Theorem 
\ref{thm:classification}, an explicit version 
for $\aone$ of Margolis's construction for killing $H^* (-, Q_1)$-cohomology 
classes is introduced. In the application, it will be 
convenient to work with modules which are bounded-above; this leads to the 
consideration  of  $DP_0$, the dual of $P_0$, and of related modules.

\begin{defn}
\label{def:R}
(Cf. \cite{bruner_Ossa}.)
 Let $R$ denote the unique (up to isomorphism) bounded-below, reduced 
$\aone$-module which is $Q_1$-acyclic and has $H^*(R, Q_0)$ one-dimensional, 
concentrated in degree $-1$.
\end{defn}

\begin{rem}
\label{rem:realizeR}
 The module $R$ can be realized topologically as the mod $2$ cohomology of the 
fibre of the transfer $\Sigma^\infty \rp^\infty \rightarrow S^0$ (recall that 
the transfer is trivial in mod $2$ cohomology) and also as 
 the fibre of a map $\rp^\infty_{-1} \rightarrow S^0$ which is non-trivial in 
mod $2$ cohomology (for instance, the composite $\rp^\infty_{-1} \rightarrow 
\rp^\infty _0 \simeq \rp^\infty \vee S^0 \rightarrow S^0$, where the 
 last map is the projection).
\end{rem}

The characterization of $R$ given in Definition \ref{def:R} implies
the following:

\begin{lem}
\label{lem:R_stab_eq}
 There are stable equivalences:
 \begin{eqnarray*}
  \Sigma^{-1} \Omega R &\simeq& R \\
  J \otimes R &\simeq & R. 
 \end{eqnarray*}
\end{lem}

There are non-split short exact sequences 
of $\aone$-modules (cf. \cite{bruner_Ossa}): 
\begin{eqnarray*}
 0 \rightarrow \field \rightarrow P_0 \rightarrow R \rightarrow 0 \\
0 \rightarrow \Omega P_0 \rightarrow \Sigma R \rightarrow \field \rightarrow 0.
\end{eqnarray*}

\begin{rem}
 These short exact sequences can be realized topologically by applying mod $2$
cohomology to the (stable) 
fibre sequences of Remark \ref{rem:realizeR}.
\end{rem}

By inspection (and as implied by Theorems \ref{thm:fundamental} and 
\ref{thm:yu}),  there is a unique non-trivial morphism
\begin{eqnarray*}
\Sigma^3 \Omega ^{-1}  J & \rightarrow & \Sigma^{-3} \Omega P_0,
\end{eqnarray*}
and this is a  monomorphism.

Similarly for the dual $DR$ of $R$, there are unique non-trivial morphisms:
\begin{eqnarray*}
 \field & \rightarrow & \Sigma^{-1} DR \\
\Sigma^{3} \Omega^{-1} J &\rightarrow & \Sigma DR 
\end{eqnarray*}
and these fit into short exact sequences:
\begin{eqnarray*}
 0 \rightarrow \field \rightarrow  \Sigma^{-1} DR  \rightarrow \Omega^{-1} DP_0 
\rightarrow 0 \\
0 \rightarrow \Sigma^3 \Omega^{-1} J \rightarrow \Sigma DR \rightarrow 
\Sigma^{-3} DP_0 \rightarrow 0.
\end{eqnarray*}

\begin{nota}
\label{nota:killing_F}
For $M$ an $\aone$-module equipped with a monomorphism $\field \hookrightarrow 
M$, let $\overline{M}$ be the module defined by the pushout of short exact 
sequences 
\[
 \xymatrix{
\field 
\ar[r]
\ar[d]
&
\Sigma^{-1} DR 
\ar[d]
\ar[r]
&
\Omega^{-1} DP_0
\ar@{=}[d]
\\
M 
\ar[r]
&
\overline{M}
\ar[r]
&
\Omega^{-1} DP_0.
}
\]
\end{nota}

\begin{lem}
\label{lem:killing_F}
 For $M$ an $\aone$-module equipped with a monomorphism $\field \hookrightarrow 
M$, there is a short exact sequence 
\[
 0 \rightarrow 
\field 
\rightarrow 
M \oplus \Sigma^{-1} DR 
\rightarrow 
\overline{M}
\rightarrow 
0
\]
and 
\begin{enumerate}
\item 
if $M$ is reduced, then so is $\overline{M}$;
\item 
if $M = M^{\geq -1}$, then $M \cong \overline{M}^{\geq -1}$;
 \item 
if $M$ is $Q_0$-acyclic, then so is $\overline{M}$;
\item 
if $\field \hookrightarrow M$ induces a non-trivial map in $H^* ( -, Q_1)$, 
then 
$H^* (\overline{M}, Q_1) \cong H^* (M, Q_1)/ \field$.
\end{enumerate}
\end{lem}

\begin{proof}
 Straightforward.
\end{proof}

\begin{nota}
\label{nota:killing_Q}
For $M$ an $\aone$-module equipped with a monomorphism $\Sigma^3 \Omega^{-1} J 
\hookrightarrow M$, let $\widetilde{M}$ be the module defined by the pushout of 
short exact sequences 
\[
 \xymatrix{
\Sigma^3 \Omega^{-1} J 
\ar[r]
\ar[d]
&
\Sigma DR 
\ar[d]
\ar[r]
&
\Sigma^{-3} DP_0
\ar@{=}[d]
\\
M 
\ar[r]
&
\widetilde{M}
\ar[r]
&
\Sigma^{-3} DP_0.
}
\]
\end{nota}

Lemma \ref{lem:killing_F} has the following counterpart:

\begin{lem}
\label{lem:killing_Q}
 For $M$ an $\aone$-module equipped with a monomorphism $\Sigma^3 \Omega^{-1} J 
\hookrightarrow M$, there is a short exact sequence 
\[
 0 \rightarrow 
\Sigma^3 \Omega^{-1} J
\rightarrow 
M \oplus \Sigma DR 
\rightarrow 
\widetilde{M}
\rightarrow 
0
\]
and 
\begin{enumerate}
\item 
if $M$ is reduced, then so is $\widetilde{M}$;
\item 
if $M = M^{\geq -1}$, then $M \cong \widetilde{M}^{\geq -1}$;
 \item 
if $M$ is $Q_0$-acyclic, then so is $\widetilde{M}$;
\item 
if $\Sigma^3 \Omega^{-1} J \hookrightarrow M$ induces a non-trivial map in $H^* 
( -, Q_1)$, then $H^* (\overline{M}, Q_1) \cong H^* (M, Q_1)/ \field$.
\end{enumerate}
\end{lem}

\begin{rem}
 In Section \ref{sect:finite_indec}, where this result is applied,  $M$ is a 
finite $\aone$-module, hence the respective 
modules $\overline{M}$, $\widetilde{M}$ are bounded-above.
\end{rem}

\section{A family of finite, indecomposable $\aone$-modules}
\label{sect:finite_indec}

The aim of this section is to classify the isomorphism classes of finite
indecomposable $\aone$-modules $M$ such that $(M|_{\eone})\red$ is
indecomposable. The guiding principle is provided by the following consequence 
of \cite[Theorem III.16.11]{adams}.

\begin{prop}
\label{prop:Adams_indec_E1}
The following conditions on a finite,
reduced $\eone$-module $L$ are equivalent:
\begin{enumerate}
 \item 
the total dimension of $H^* (L, Q_0) \oplus H^* (L, Q_1)$ is two;
\item
$L$ is a reduced indecomposable module. 
\end{enumerate}
\end{prop}

Hence, for $M$ as above, there are three cases to consider:
\begin{enumerate}
 \item 
$\dim H^*(M, Q_0) = 1 = \dim H^*(M, Q_1)$: by the results of Adams and Priddy
\cite{adams_priddy}, this implies that $M$ is in the Picard group, $\pica$, 
hence these modules 
are classified by the Picard group;
\item 
$H^* (M, Q_1) =0$ and $\dim H^* (M, Q_0)=2$: this case is easily understood 
(note the relative simplicity of Proposition \ref{prop:tech2} compared to 
Proposition \ref{prop:technical}) - see Proposition 
\ref{prop:classify_Q1_acylic};
\item 
$H^* (M, Q_0) =0$ and $\dim H^* (M, Q_1) =2$: the most interesting case.
\end{enumerate}

Recall the module $R$ of Definition \ref{def:R} and the form of the module 
$\aone \otimes_\azero \field$
  (see Remark \ref{rem:A1A0}). By the results 
of Section \ref{sect:basic}, there is 
a canonical inclusion of $\Sigma^{-1} \aone \otimes_\azero \field$ and, 
moreover, this fits into a non-split short exact sequence:
\[
 0
 \rightarrow 
 \Sigma^{-1} \aone \otimes_\azero \field
 \rightarrow 
 R
 \rightarrow 
 \Sigma^4 R 
 \rightarrow 
 0
\]
(cf. \cite{bruner_Ossa}). Recursively one obtains an increasing filtration 
\[
0=f_0 R \subset f_1 R =  \Sigma^{-1} \aone \otimes_\azero \field
\subset 
f_2 R \subset 
\ldots\subset f_i R \subset \ldots R, 
\]
where $f_{i+1}R/ f_i R \cong \Sigma^{4i-1} \aone \otimes_\azero \field$, each 
$f_i R$ is $Q_1$-acyclic and, for $i \geq 1$, 
\[
 H^* (f_i R, Q_0) \cong \left\{ 
 \begin{array}{ll}
  \field & i = -1, 4i
  \\
  0 &\mathrm{otherwise.}
 \end{array}
\right.
\]
The classification in the $Q_1$-acyclic case is as follows:

\begin{prop}
\label{prop:classify_Q1_acylic}
 Let $M\neq 0$ be a reduced, finite $\aone$-module. 
 The following conditions are equivalent:
 \begin{enumerate}
  \item 
  $M$ is  $Q_1$-acyclic and $\dim H^* (M, Q_0) =2$;
  \item 
  $M \cong \Sigma^{d+1} f_i R$, where $H^* (M, Q_0)$ is non-zero in degrees  
 $d$, $4i +d + 1$ for integers $d, 1 \leq i$.
 \end{enumerate}
\end{prop}

\begin{proof}
The implication (2)$\Rightarrow$(1) is straightforward, hence consider the 
converse.

The hypotheses imply that $M|_{E(1)}\red$ is indecomposable (by Proposition 
\ref{prop:Adams_indec_E1}) with the form of 
a $Q_1$-acyclic lightning flash, which gives  
 that $H^* (M, Q_0)$ is non-zero precisely in degrees of the form  $d$, $4i +d 
+ 
1$, for 
$d \in \zed$ and $i \in \nat$. The case $i=0$ is excluded by Proposition 
\ref{prop:tech2}. 

 The result can be proved by induction upon $i$. For $i=1$, Corollary 
\ref{cor:connectivity_Q1_acyclic} and its dual imply that $M$ is concentrated 
in degrees $[d,d+3]$. 
 Proposition \ref{prop:tech2} provides an embedding  $\Sigma^{d+1} \aone 
\otimes_\azero \field
 \hookrightarrow M$ and it is straightforward to see that this is a stable 
isomorphism. 
 
 For the inductive step, $i>1$, a similar argument provides a short exact 
sequence 
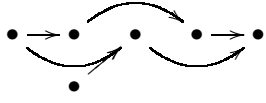
 \[
  0
  \rightarrow 
  \Sigma^{-1} \aone \otimes_\azero \field
 \rightarrow 
 M
 \rightarrow 
 M'
 \rightarrow 
 0
 \]
where $M'$ is reduced and  has $Q_0$-Margolis cohomology in degrees $d+4$, $4i 
+d + 1$, hence is isomorphic to $ \Sigma^{d+4+1} f_{i-1} R$, by the inductive 
hypothesis. There is a unique non-trivial extension of this form, 
which corresponds to $ \Sigma^{d+1} f_i R$, as required.
 \end{proof}

Now consider the $Q_0$-acyclic case; the associated $\eone$-module
$(M|_{\eone})\red$ therefore has the following form:
\[
 \xymatrix{
&
\Diamond 
&&
\bullet 
&&
\ldots 
&&
\bullet 
&&
\bullet
\\
\bullet
\ar[ur]
\ar[urrr]
&&
\bullet
\ar[ur]
\ar@{.}[urrr]
&&
\ldots
\ar@{.}[ur]
\ar@{.}[urrr]
&&
\bullet
\ar[ur]
\ar[urrr]
&&
\Diamond
\ar[ur]
&.
}
\]
The $Q_1$-Margolis cohomology classes are represented by the top left hand
generator and the bottom right hand generator, indicated above by 
$\Diamond$; the case where the total
dimension of 
$(M|_{\eone})\red$ is two is exceptional (in this case, 
the bottom generator represents the $Q_1$-cohomology class of lowest degree). 
This gives:

\begin{lem}
\label{lem:Q0-acyclic_lightning}
 Let $M$ be a finite $Q_0$-acyclic $\aone$-module such that $(M|_{\eone})\red 
\neq 0$ and is indecomposable, then $H^* (M, Q_1)$ has total dimension $2$ and 
is concentrated
in degrees $a, a+ 2s-3$, for $a \in \zed$ where $1 \leq s \in \nat$ is the 
total 
dimension
of the socle of $(M|_{\eone})\red$.
\end{lem}

First consider the family of modules which occurs in the exceptional case 
$s=1$.

\begin{nota}
 Write $Z$ for the $\aone$-module defined by the non-split extension
\[
 0
\rightarrow 
\field 
\rightarrow 
\Sigma^{-1} Z 
\rightarrow 
\Sigma^{-1} \field 
\rightarrow
0.
\]
Note that $Z$ is  not self dual, since $DZ \cong \Sigma Z$. 
\end{nota}

\begin{rem}
 The $\aone$-module $Z$ is the mod $2$ cohomology of the desuspended  Moore 
spectrum 
 $\Sigma ^{-1} S^0/2$. The behaviour of duality on $Z$ corresponds to that of 
Spanier-Whitehead duality on $\Sigma ^{-1} S^0/2$.
\end{rem}

\begin{lem}
\label{lem:Omega_Z}
 There are stable isomorphisms:
\begin{eqnarray*}
\Omega^4 Z & \simeq & \Sigma^{12} Z \\
\Omega ^2 Z &\simeq &\Sigma ^6 J \otimes Z,  
\end{eqnarray*}
$\Omega Z$ occurs in the non-split short exact sequence
\[
 0 
\rightarrow
\Sigma ^3 J \rightarrow 
\Omega Z 
\rightarrow 
\Sigma^2\field 
\rightarrow 
0,
\]
and $\Omega^{-1} Z \cong \Sigma^{-1} D (\Omega Z)$.
\end{lem}

\begin{proof}
The first statement is an application  of Proposition \ref{prop:Q_0-acyclic}.  
The 
structure of $\Omega Z$ is a straightforward
computation
and that of $\Omega^{-1}Z$ follows by duality.
\end{proof}

The structure of $\Omega Z$ is:
\[
\xymatrix @C=1pc @R=.75pc{
\bullet 
\ar@/_1pc/[rr]
\ar[r]
&
\bullet 
\ar@/^1pc/[rr]
&
\bullet
\ar@/_1pc/[rr]
&
\bullet 
\ar[r]
&
\bullet
\\
&
\bullet
\ar[ru]
}
\]
with the generators placed in the appropriate degrees.

\begin{rem}
\label{rem:attach-Joker}
 Where a Joker occurs as a subquotient of an $\aone$-module, as in the case 
$\Omega Z$ above, the diagram will be simplified by representing the Joker by 
$\circ$. In the above case, this gives 
 $
  \bullet \rightarrow \circ.
 $
\end{rem}

\begin{rem}
\label{rem:quotient_s=1}
 The module $Z$ is a quotient of $P_0$: there is a short exact sequence
of $\aone$-modules:
\[
 0
\rightarrow 
\Sigma^{-1} \Omega P_0 
\rightarrow 
P_0
\rightarrow 
Z 
\rightarrow 
0.
\]
Similarly, $\Omega Z$ occurs as a quotient of $\Omega^2
P_0$ (up to suitable suspension) and $\Omega^{-1} Z $ as a quotient of
$\Omega^{-1} P_0$.
\end{rem}

\begin{prop}
\label{prop:truncate_P}
 Let $P$ be a $Q_0$-acyclic, bounded-below, reduced $\aone$-module such that
$H^* (P, Q_1)$ is one-dimensional, concentrated in degree zero, and let $2 \leq
s$ be an integer.
 Then 
 \begin{enumerate}
  \item 
the quotient $P ^{\leq 2(s-1)}$ is a reduced, finite $\aone$-module which 
is $Q_0$-acyclic and has $Q_1$-Margolis cohomology of total dimension $2$,
concentrated  in degrees $\{ 0 , 2s-3 \}$;
\item 
there is a short exact sequence of $\aone$-modules:
\[
 0 \rightarrow
 \Sigma^{2s} (\Sigma^{-3} \Omega) ^{\overline{i+s}} P_ 0 
 \rightarrow 
 P 
 \rightarrow 
 P ^{\leq 2(s-1)}
 \rightarrow 
 0
\]
where, for $P \cong (\Sigma^{-3} \Omega )^i P_0$, $ i \in \{0, 1, 2, 3 \}$ and 
$\overline{i+s} \in \{0, 1\}$ is the residue of $i+s$ modulo $2$.
 \item
 \label{item:st_iso_classes}
 the stable isomorphism classes of the modules 
 $$\big\{ J^{\otimes \epsilon} 
\otimes ((\Sigma^{-3}\Omega)^iP_0 )^{\leq 2(s-1)} | \epsilon \in \{0, 1 \}, i 
\in \{0, 1, 2,3 \}  \big\}
$$  are pairwise distinct; 
 \item 
 under the action of $\Sigma^{-3}\Omega$, the stable isomorphism classes of 
(\ref{item:st_iso_classes}) form 
two distinct 
orbits of cardinal four, generated by $(P_0)^{\leq 2(s-1)}$ 
and $((P_0 \otimes J)\red)^{\leq 2(s-1)}$.
\end{enumerate}
\end{prop}

\begin{proof}
 By Yu's theorem (Theorem \ref{thm:yu}), the first two parts of the result can 
be read  off by inspection from the structure of the representatives of the 
isomorphism classes
of such modules (using the details provided by Bruner \cite{bruner_Ossa}, in 
particular \cite[Theorem 4.6]{bruner_Ossa}). Similarly, the proof that the 
stable isomorphism classes are pairwise distinct 
is straightforward.

The cardinality of the orbits of the stable isomorphism classes under $ 
\Sigma^{-3}\Omega$ divides four, 
by Proposition \ref{prop:Q_0-acyclic}, since each of the modules considered is 
$Q_0$-acyclic;  moreover, since $ \Sigma^{-6}\Omega^2 P_0 \simeq J \otimes 
P_0 \not \simeq P_0$,  the orbits have cardinal four.

Finally, it is straightforward to check that the given elements lie in distinct 
orbits. 
\end{proof}

\begin{rem}
\ 
\begin{enumerate}
 \item 
 The condition $s \geq 2$ is required due to the possible presence of a
sub-quotient isomorphic to (a suspension of) the Joker in low degree.
 The appropriate quotient in the case $s=1$ given in  Remark 
\ref{rem:quotient_s=1} cannot be defined  simply by truncation.
\item
$\aone$-modules of the form considered in Proposition \ref{prop:truncate_P} 
 already appear in the literature. First examples are given by the cohomology 
of truncated projective spaces; in his thesis \cite{davis}, Don Davis 
considered a related family of $\aone$-modules (see \cite[Definition 
3.6]{davis}) and \cite[Lemma 3.8]{davis}  can be interpreted as describing the 
action of the Picard group.
\end{enumerate}
\end{rem}

\begin{figure}
 \caption{One orbit under $\Sigma^{-3}\Omega$ for $s=2$}
 \label{fig:orbit_s=2}
 \[
  \xymatrix @R=.75pc @C=.75pc{
  \bullet\ar@/^1pc/[rr]
  \ar[r]
  &
  \bullet 
  &
  \bullet 
  \ar[r]
  &
  \bullet
  &&
  \bullet
  \ar[r]
  &
  \bullet 
  \ar@/^1pc/[rr]
  &
  \bullet 
  \ar[r]
  &
  \bullet 
  \\
  \\
  &&&&\circlearrowright&&&
  \\
  \\
  \circ
  \ar[r]
  &
  \bullet 
  \ar@/^1pc/[rr]
  &
  \circ
  \ar[r]
  &
  \bullet 
  &
  &
  \bullet
  \ar@/^1pc/[rr]
  \ar[r]
  &
  \circ
  &
  \bullet
  \ar[r]
  &
  \circ   
  }
 \]
  \end{figure}
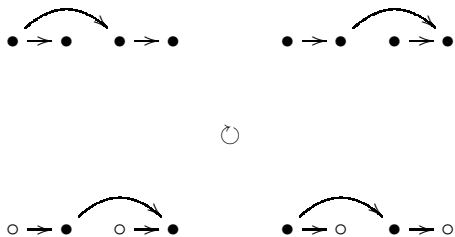

\begin{exam}
\label{exam:orbits}
 The orbits under $\Sigma^{-3} \Omega$ are easily understood. For example, in 
the case $s=2$, one of the orbits is illustrated in Figure \ref{fig:orbit_s=2}.
(See Remark \ref{rem:attach-Joker} for the notation for attaching a Joker.)

The  pattern in lowest degrees corresponds to that of the orbit $(\Sigma^{-3} 
\Omega)^i P_0$, whereas that in highest degrees is dual, hence cycles in 
the opposite order. The description of the second orbit is similar. 

The behaviour depends on the parity of $s$; to illustrate this, see Figure 
\ref{fig:orbit_s=3}, which shows one of the orbits for 
 $s=3$ (the second is given below in Figure \ref{fig:second_orbit_s=3}).

In both cases, $\Sigma^{-6} \Omega^2$ operates as $J \otimes -$, 
switching $\bullet \longleftrightarrow \circ$ in the degrees which correspond 
to the $H^* (-,Q_1)$ cohomology classes; this gives a form of symmetry across 
the diagonal. Thus, to understand the orbit, it is sufficient to calculate the 
action of $\Sigma^{-3} \Omega$. 
\end{exam}

\begin{figure}
 \caption{One orbit under $\Sigma^{-3}\Omega$ for $s=3$}
 \label{fig:orbit_s=3}
\[
  \xymatrix@R=.75pc @C=.75pc{
  \bullet\ar@/^1pc/[rr]
  \ar[r]
  &
  \bullet 
  &
  \bullet 
  \ar[r]
  &
  \bullet 
  \ar@/^1pc/[rr]
  &
  \bullet \ar[r]
  &
  \bullet
  &&
  \bullet
  \ar[r]
  &
  \bullet 
  \ar@/^1pc/[rr]
  &
  \bullet 
  \ar[r]
   \ar@/_1pc/[rr]
  &
  \bullet
  &
  \bullet 
  \ar[r]
  &
  \circ
  \\
  \\
  &&&&&&\circlearrowright&&&
  \\
  \\
  \circ
  \ar[r]
  &
  \bullet 
  \ar@/^1pc/[rr]
  &
  \bullet
  \ar[r]
  \ar@/_1pc/[rr]
  &
  \bullet 
  &
  \bullet 
  \ar[r]
  &
  \bullet
  &
  &
  \bullet\ar@/^1pc/[rr]
  \ar[r]
  &
  \circ
  &
  \bullet
  \ar[r]
  &
  \bullet 
  \ar@/^1pc/[rr]
  &
  \circ \ar[r]
  &
  \bullet
  }
 \]
 \end{figure}
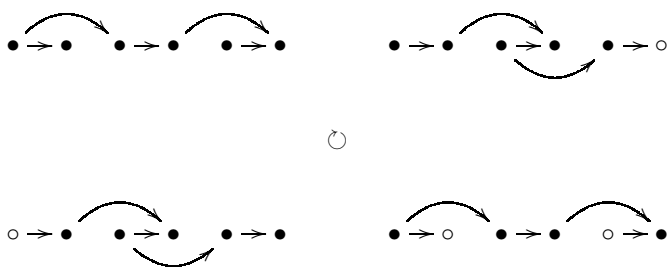

\begin{thm}
\label{thm:classification}
 Let $M$ be a finite, reduced $\aone$-module such that $H^* (M, Q_0) =0$ and 
$H^* (M, Q_1)$ has total dimension two, concentrated in degrees $0, 2s-3$, for 
$1 \leq s \in \nat$  the total dimension of the socle of $(M|_{\eone})\red$; 
\begin{enumerate}
 \item 
if $s=1$, then $M \simeq (\Sigma^{-3} \Omega)^t Z $, for some $t \in \{0, 1,2,3 
\}$; 
\item 
if $s \geq 2$, then $M \simeq J ^{\otimes \epsilon} \otimes P^{\leq 2(s-1)}$, a 
module of the form given in Proposition \ref{prop:truncate_P}.
\end{enumerate}
\end{thm}

\begin{proof}
The cases $s=1$ and $s=2$ require separate treatment; $s=1$ is the exceptional 
case and the condition $H^1 (M, Q_1) =0$ in Corollary \ref{cor:embedding} has 
to  be worked around for $s=2$.

First consider the case $s>2$; by Corollary \ref{cor:embedding} (replacing $M$ 
by $(M \otimes J )\red$ if necessary, which leads to the factor $J^{\otimes 
\epsilon} \otimes -$ with $\epsilon =1$), there exists a monomorphism of one of 
the following forms
\begin{enumerate}
 \item 
 $\field \hookrightarrow M$
 \item
 $\Sigma^3 \Omega^{-1} J \hookrightarrow M$
\end{enumerate}
and $M = M ^{\geq -1}$. 

Consider the first case; forming the finite-type,  bounded-above module 
$\overline{M}$ of Notation \ref{nota:killing_F}, Lemma \ref{lem:killing_F} 
shows 
that $ \overline{M}$ is finite-type, reduced, $Q_0$-acyclic, with $H^* (M, 
Q_1)$ 
one-dimensional, concentrated in degree $2s-3$. Thus, the dual of Theorem 
\ref{thm:yu} identifies the isomorphism type of the  module $\overline{M}$ and, 
again by Lemma \ref{lem:killing_F}, $M$ is recovered as the submodule 
$\overline{M}^{\geq -1}$. The required conclusion follows by dualizing.

The second case is analogous, {\em mutatis mutandis}, using $\widetilde{M}$ of 
Notation \ref{nota:killing_Q} together with Lemma \ref{lem:killing_Q}.

In the case $s=2$, the same strategy can be applied, once the appropriate 
result 
corresponding to Corollary \ref{cor:embedding} has been established. In this 
case, the underlying $\eone$-module $(M|_{\eone})\red$ is of the following form:
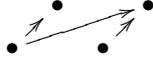
\[
 \xymatrix @R=.5pc @C=.5pc{
 & 
 \bullet &&
 \bullet 
 \\
 \bullet 
 \ar[ur]
 \ar[urrr]
 &&\bullet
 \ar[ur]
 }
\]
with $Q_1$-Margolis cohomology in degrees $0, 1$. Proposition 
\ref{prop:technical} implies that $M = M^{\geq -3}$; if $M^{-3} \neq 0$, then 
the argument of the proof of Proposition \ref{prop:technical} provides a 
monomorphism $\Sigma^3 \Omega^{-1} \field \hookrightarrow M$
 inducing a monomorphism on $H^0 (-, Q_1)$; replacing $M$ by $M \otimes J$, one 
obtains a monomorphism $ \Sigma^3 \Omega^{-1} J \hookrightarrow (M \otimes J)$ 
inducing a monomorphism  on $H^0 (-, Q_1)$. The argument now proceeds as above.
 
 Now consider the case $M^{-3}=0$ and $M^{-2} \neq 0$ then, as in Proposition 
\ref{prop:technical}, $Sq^2 Sq^2$ acts injectively on $M^{-2}$. Consider $Sq^1 
Sq^2 (M^{-2}) \subset M^1$, which clearly lies in $\ker(Q_0) \cap \ker (Q_1)$, 
since $M$ is reduced; by hypothesis on the 
 form of $(M|_{\eone})\red$, this is only possible if $Sq^1 Sq^2 (M^{-2})$ lies 
in the image of $Q_0 Q_1$, which is excluded by the hypothesis that $M^{-3} = 
0$. Hence there exists a monomorphism $J \hookrightarrow M$ and one proceeds as 
before.

 In the remaining case, $M= M^{\geq -1}$; by duality, one can further reduce to 
the case where $M$ is concentrated in degrees $[-1, 2]$. In this case 
$M|_{\eone}$ is necessarily $\eone$-reduced, of total dimension four, and there 
are the two possibilities:

\smallskip

 
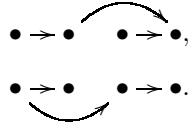
\[
  \xymatrix@R=.75pc @C=.75pc{
  \bullet 
  \ar[r]
  &
  \bullet 
  \ar@/^1pc/[rr]
  &
  \bullet 
  \ar[r]
  &
  \bullet, 
  \\
  \bullet 
  \ar@/_1pc/[rr]
  \ar[r]
  &
  \bullet 
  &
  \bullet 
 \ar[r]
 &
 \bullet.
 }
 \]
 
 \ 
 \smallskip
 
 \noindent
These cases provide the isomorphism classes given in the statement. 

It remains to consider the case $s=1$, so that $(M|_{\eone})\red \cong Z 
|_{\eone}$. A similar strategy applies; Proposition \ref{prop:technical} 
implies 
that $M$ is concentrated in degrees $[-4 ,3]$; the first step is to show that 
$M^{-4}=0$ (so that $M^3=0$ by a dual argument). This follows as above, since 
$Sq^1 Sq^2 (M^{-4})$ lies in $\ker(Q_0) \cap \ker (Q_1)$, which is zero. Thus 
$M$ 
is concentrated in degrees $[-3,2]$. 

Suppose $M^{-3} \neq 0$, then $Sq^2 Sq^2$ acts injectively on $M^{-3}$ and 
$Sq^2 
(M^{-3}) \cap \ker (Q_0) \subset M^{-1} $ lies in $\ker (Q_0) \cap \ker (Q_1)$, 
which is zero (as above), hence $Sq^1 Sq^2 : M^{-3} \hookrightarrow M^0$. There 
are two cases:

\begin{enumerate}
 \item 
 If $Sq^2 Sq^1 Sq^2 $  is not injective on $M^{-3}$, then one obtains  an 
inclusion 
$$\Sigma^3 \Omega ^{-1} Z \hookrightarrow M$$
which induces an  isomorphism on Margolis cohomology groups, hence is an 
isomorphism.
 \item 
Otherwise, by also invoking the dual argument, we may assume that   $Sq^2 Sq^1 
Sq^2$ induces an isomorphism $M^{-3} \stackrel{\cong}{\rightarrow} 
M^2$ and, counting Margolis cohomology classes, $\dim M^{-3} =1$. From this it 
is straightforward to prove that $J \otimes Z$ and $M$ have the same underlying 
graded vector space; to check that they are isomorphic, 
it suffices to show that $M$ fits into a non-trivial extension of the form
\[
 0
 \rightarrow 
 J 
 \rightarrow M
 \rightarrow 
 \Sigma^{-1} J 
 \rightarrow 
 0,
\]
since there is a unique non-trivial extension of this form.

To exhibit the submodule $J$, using the above arguments, one 
can show that $\ker (Sq^1 Sq^2)$ is one-dimensional in degree $-2$ and the 
non-trivial class generates a submodule isomorphic to $J$. The quotient module 
is checked to be isomorphic to $\Sigma^{-1} J$, as required.
\end{enumerate}

In the remaining cases, $M$ is concentrated in degrees $[-2, 1]$, hence is 
isomorphic to $Z$, by consideration of $M|_{\eone}$.
\end{proof}

\section{An interpretation of the families using the Picard group}
\label{sect:picard_interpret}

The presentation of the families arising in Section \ref{sect:finite_indec} was 
guided by the method of proof of Theorem \ref{thm:classification}.
A convenient choice of generators for the orbits under the action 
of the Picard group (see Proposition \ref{prop:truncate_P}) can be given by 
removing the 
lowest 
dimensional class of suitable elements of the Picard group.

For $1\leq k \in \nat$, the Picard group elements $\Sigma^{-(k+1)} \Omega^{k+1} 
\field$ and $\Sigma^{-(k+1)} \Omega^{k+1} J$ have isomorphic Margolis 
cohomology 
groups, with $H^* (-,Q_0)$ concentrated in degree zero and $H^* (-,Q_1)$ in 
degree $2(k+1)$.

\begin{rem}
The case which would correspond to $s=1$ in Theorem \ref{thm:classification} is 
 exceptional, hence is excluded in this section.
\end{rem}

\begin{lem}
\label{lem:Picard_construction_Q0-acyclics}
 If $k \geq 1$, for $N= N_{k,\epsilon}: = \Sigma^{-(k+1)} \Omega^{k+1} J 
^{\otimes \epsilon}$, $\epsilon \in \{0,1\}$, 
 \begin{enumerate}
  \item 
  $N = N^{\geq 0}$ and $N^0 = \field$; 
  \item 
  the cyclic submodule generated by $N^0$ is isomorphic to $\aone \otimes 
_\azero \field$ ; 
 \item
 there is a unique non-trivial morphism $N_{k, \epsilon} \rightarrow \field$ 
which represents a non-trivial class in $[\Sigma^{-(k+1)} \Omega^{k+1} J 
^{\otimes \epsilon}, \field ] \cong \ext^{k+1,k+1}_\aone (\field, J^{\otimes 
\epsilon})$
 and this induces an isomorphism in $H^* (-, Q_0)$; 
 \item 
 for $A_{k, \epsilon}$ defined by the short exact sequence 
 \[
  0
  \rightarrow 
  A_{k, \epsilon}
  \rightarrow 
  \Sigma^{-(k+1)} \Omega^{k+1} J ^{\otimes \epsilon}
  \rightarrow 
  \field 
  \rightarrow 
  0,
 \]
$A_{k, \epsilon}$ is a reduced, finite  $\aone$-module which 
is $Q_0$-acyclic and has Margolis cohomology groups $H^* (-, Q_1)$ isomorphic 
to 
$\field$ concentrated in degrees $\{ 3, 2(k+1) \}$.  
 \end{enumerate}
\end{lem}

\begin{proof}
The embedding of $\aone \otimes_\azero \field$ in $\Sigma^{-(k+1)} 
\Omega^{k+1} 
J ^{\otimes \epsilon}$ is provided by Proposition \ref{prop:tech2} which, 
together with Proposition \ref{prop:technical},  also suffices to show the 
one-dimensionality of $N^0$. 
 The non-triviality of the stable class of the  projection onto degree zero 
follows by the dual of Lemma \ref{lem:stable_hom_reduced}, since $N$ is reduced 
(by definition of $\Omega$); the relevant $\ext$ group is $\field$ (see Section 
\ref{sect:ext}). 
 That the projection induces an isomorphism in $H^* (-, Q_0) $ is clear. 
 
 The final statement follows from the fact that $N_{k,\epsilon}$ is finite and 
from the calculation of the Margolis cohomology from the associated 
long exact sequences. 
\end{proof}

For $1 \leq k \in \nat$, the above Lemma provides two reduced $\aone$-modules 
$A_{k,0}$, $A_{k,1}$ which are not stably isomorphic. 
After applying $\Sigma^{-3}$, so that the $Q_1$-Margolis cohomology 
classes are in degrees $\{0, 2k-1 \}$,  these fit into the family of modules 
 classified by Theorem \ref{thm:classification}, for $s=k+1$ ($s\geq 2)$.

 The modules $A_{k, \epsilon}$ can be viewed as `normalized' choices of orbit 
representatives using the existence of  an inclusion of the question mark 
complex to provide the normalization; this fits into the philosophy espoused by 
Corollary \ref{cor:embedding}. 
 
\begin{prop}
\label{prop:A_orbits_normalization}
 For $k \geq 1$ and $\epsilon \in \{0,1 \}$, 
 \begin{enumerate}
  \item 
  there is a monomorphism 
  $
\Sigma^6 \Omega ^{-1} J \hookrightarrow A_{k,\epsilon}
 $
and
\[
\hom_{\aone} (\Sigma^6 \Omega ^{-1} J , (\Omega \Sigma^{-3})^t A_{k,\epsilon} 
)=
\left\{
\begin{array}{ll}
 \field & t \equiv 0 \mod (4) \\
 0& \mathrm{otherwise};
\end{array}
\right.
\]
\item 
$A_{k,0} \not \simeq (\Omega \Sigma^{-3})^t A_{k,1}$ for  $t \in \zed$;
\item 
under the action of  $\Sigma^{-3} \Omega$,
the modules   $\Sigma^{-3} A_{k,0}$,  $\Sigma^{-3}A_{k,1}$ generate   the set 
of stable isomorphism classes 
 of modules classified by Theorem \ref{thm:classification}, for $s=k+1$.
 \end{enumerate}
\end{prop}

\begin{proof}
The embedding  $\aone \otimes_\azero \field \hookrightarrow \Sigma^{-(k+1)} 
\Omega^{k+1} J ^{\otimes \epsilon}$ of Lemma 
\ref{lem:Picard_construction_Q0-acyclics} restricts to the  monomorphism $ 
\Sigma^6 \Omega ^{-1} J \hookrightarrow 
A_{k,\epsilon}$ and this is the unique non-trivial such morphism. The 
triviality of $\hom_{\aone} (\Sigma^6 
\Omega ^{-1} J , (\Omega \Sigma^{-3})^t A_{k,\epsilon} )$ in other cases can 
either be proved by the methods of Section \ref{sect:basic}  
or by using the calculations of Section \ref{sect:ext} as follows. Namely, by 
Lemma 
\ref{lem:question_stable}, one can work with stable groups $\stext$ (see 
Section 
\ref{sect:ext}); the triviality of the relevant groups follows from 
Theorem \ref{thm:calculate_stext_Ake} except in the case of the fundamental
class ($\mu$ in the notation of Section 
\ref{sect:ext}), corresponding to $t \equiv 0 \mod (4)$. 

The second statement follows as an immediate consequence; in order to obtain a 
non-trivial stable map in $[\Sigma^6 \Omega ^{-1} J , (\Omega \Sigma^{-3})^t 
A_{k,1} ]$, one must take $t \equiv 0 \mod (4)$, so that  $(\Omega 
\Sigma^{-3})^t A_{k,1}\simeq A_{k,1}$. However, $A_{k,0}$ and $A_{k,1}$ are not 
stably isomorphic.

The final statement then follows from Theorem \ref{thm:classification} and 
Proposition \ref{prop:truncate_P}.
\end{proof}

\begin{prop}
 \label{prop:identify_A}
 For $k \geq 1$ and $\epsilon \in \{0,1\}$,  there are isomorphisms:
 \begin{eqnarray*}
 A_{k,\epsilon} & \cong & 
 \Big\{ 
 \Sigma^{4\epsilon}
 D\big( 
 (\Omega^{-1} \Sigma)^{k+1 - 2\epsilon} P_0  
  \big) 
\Big \}^{\geq 2}
  \\
  &\cong &
 \Big\{  \Sigma^{4\epsilon} (\Omega \Sigma^{-1})^{k+1 - 2\epsilon} D 
  P_0 
  \Big \}^{\geq 2}.
 \end{eqnarray*}
\end{prop}

\begin{proof}
 For $\epsilon \in \{0, 1 \}$, since $A_{k,\epsilon}$ and $ (\Omega 
\Sigma^{-1})^{k+1}J^{\otimes \epsilon}$ are finite $\aone$-modules, results 
dual 
  to those of Section \ref{sect:applications} imply that  the 
injection 
$A_{k,\epsilon} \hookrightarrow (\Omega \Sigma^{-1})^{k+1}J^{\otimes \epsilon}$ 
induces
 \begin{eqnarray*}
   A_{k,\epsilon} \simeq A_{k,\epsilon} \otimes DP_0 \rightarrow (\Omega 
\Sigma^{-1})^{k+1}J^{\otimes \epsilon} \otimes DP_0 &\simeq& \Sigma^{4 
\epsilon} 
(\Omega \Sigma^{-1})^{k+1 - 2 \epsilon} DP_0\\
   &\cong&   \Sigma^{4 \epsilon} D ( (\Omega^{-1} \Sigma)^{k+1 - 2 \epsilon} 
P_0),
 \end{eqnarray*}
which gives an isomorphism on the top $Q_1$-Margolis cohomology class. 

Since $A_{k,\epsilon} =(A_{k,\epsilon})^{\geq 2}$, this yields a morphism: 
\[
 A_{k,\epsilon} \rightarrow \big\{  \Sigma^{4 \epsilon} D ( (\Omega^{-1} 
\Sigma)^{k+1 - 2 \epsilon} 
P_0)\big\}^{\geq 2}.
\]
By inspection of the underlying $\eone$-modules and the fact that it is 
non-trivial in $Q_1$-Margolis homology, one sees that 
the morphism induces an isomorphism on both $Q_0$- and $Q_1$-Margolis cohomology 
groups, hence is a stable isomorphism. 
Since both modules are reduced (by construction), it is an isomorphism.
\end{proof}

\begin{exam}
 The behaviour of the families $A_{k,1}$ and $A_{k,0}$ can be understood by 
considering the cases $1 \leq k \leq 4$ (see Figure \ref{fig:Ak1} for $A_{k,1}$ 
and Figure \ref{fig:Ak0} for $A_{k,0}$). Namely, the module $A_{k+4, \epsilon}$ 
is obtained from $A_{k,\epsilon}$ obtained by 
extending (and shifting degrees) by

\[
 \xymatrix@R=.5pc @C=1pc{
  \bullet \ar[r] & \bullet \ar@/^1pc/[rr] & \bullet \ar[r]\ar@/_1pc/[rr] 
&\bullet& \bullet \ar[r] & \bullet \ar@/_1pc/[rr]& \bullet \ar[r] 
\ar@{.>}@/^1pc/[rr]& \bullet&  *+[F-:<3pt>]{ A_{k,\epsilon}} &,
 }
\]

\medskip
\noindent
where the dotted $Sq^2$ hits the lowest-dimensional class.

Similarly, observe   that $A_{1,1}$ occurs as a subobject of $A_{3,0}$, as 
predicted by Proposition \ref{prop:identify_A}; likewise $A_{1,0}$ occurs as a 
subobject of $A_{3,1}$.  
This behaviour is general and corresponds to (a shift of) an extension by
\medskip
\[
 \xymatrix@R=.5pc @C=1pc{
 \bullet \ar[r] & \bullet \ar@/_1pc/[rr]& \bullet \ar[r] 
\ar@{.>}@/^1pc/[rr]& \bullet& *+[F-:<3pt>]{ A_{k,\epsilon}}& .
 }
\]

\smallskip
The composition of two such extensions assures the passage $A_{k,\epsilon} 
\rightsquigarrow A_{k+4,\epsilon}$, as is evident from
\[
 \xymatrix@R=.5pc @C=1pc{
  \bullet \ar[r] & \bullet \ar@/^1pc/[rr] & \bullet \ar[r]\ar@{.>}@/_1pc/[rr] 
&\bullet& \bullet \ar[r] & \bullet \ar@/_1pc/[rr]& \bullet \ar[r] 
\ar@{.>}@/^1pc/[rr]& \bullet&  *+[F-:<3pt>]{ A_{k,\epsilon}}&.
 }
\]
\end{exam}
 
 \medskip

\begin{figure}
\caption{The family $A_{k,1}$ for $1\leq k \leq 4$}
\label{fig:Ak1}
\[
\xymatrix@R=.5pc @C=1pc{
 A_{1,1} && 
 \bullet \ar[r] & \bullet \ar@/^1pc/[rr] & \bullet \ar[r] &\bullet 
 \\
 \\
 A_{2,1} && 
 \bullet \ar[r] & \bullet \ar@/^1pc/[rr] & \bullet \ar[r]\ar@/_1pc/[rr] 
&\bullet& \bullet \ar[r] & \circ  
 \\
 \\
 A_{3,1}&&
 \bullet \ar[r] & \bullet \ar@/^1pc/[rr] & \bullet \ar[r]\ar@/_1pc/[rr] 
&\bullet& \bullet \ar[r] & \bullet \ar@/_1pc/[rr]& \circ \ar[r] & \bullet  
 \\
 \\
 A_{4,1}&&
 \bullet \ar[r] & \bullet \ar@/^1pc/[rr] & \bullet \ar[r]\ar@/_1pc/[rr] 
&\bullet& \bullet \ar[r] & \bullet \ar@/_1pc/[rr]& \bullet \ar[r] 
\ar@/^1pc/[rr]& \bullet& \bullet \ar[r] & \bullet   
 }
\]
\end{figure}

\begin{figure}
\caption{The family $A_{k,0}$ for $1\leq k \leq 4$}
\label{fig:Ak0}
\[
\xymatrix@R=.5pc @C=1pc{
 A_{1,0} && 
\bullet \ar[r] & \bullet \ar@/^1pc/[rr] & \circ \ar[r] &\bullet 
 \\
 \\
 A_{2,0}&& 
 \bullet \ar[r] & \bullet \ar@/^1pc/[rr] & \bullet \ar[r]\ar@/_1pc/[rr] 
&\bullet& \bullet \ar[r] & \bullet  
\\
\\
 A_{3,0}&& 
 \bullet \ar[r] & \bullet \ar@/^1pc/[rr] & \bullet \ar[r]\ar@/_1pc/[rr] 
&\bullet& \bullet \ar[r] & \bullet \ar@/_1pc/[rr]& \bullet \ar[r] & \bullet  
 \\
 \\
 A_{4,0}&& 
 \bullet \ar[r] & \bullet \ar@/^1pc/[rr] & \bullet \ar[r]\ar@/_1pc/[rr] 
&\bullet& \bullet \ar[r] & \bullet \ar@/_1pc/[rr]& \bullet \ar[r] 
\ar@/^1pc/[rr]& \bullet& \bullet \ar[r] & \circ   
 }
 \]
\end{figure}

The behaviour of duality is straightforward:

\begin{prop}
\label{prop:duality_A}
 For $k \in \nat$ and $\epsilon \in \{ 0,1 \}$, there is a stable isomorphism:
 \[
  A_{k,\epsilon} \simeq
  \Sigma^{-(k+1+6\epsilon)} \Omega ^{k+2 + 2\epsilon } DA_{k, \epsilon}.
 \]
\end{prop}

\begin{proof}
 This is a standard argument, using duality at the level of the stable module 
category, which is a tensor triangulated category. 
 The dual of the defining exact sequence is of the form
 \[
  0 
  \rightarrow 
  \field 
  \rightarrow 
  \Sigma^{(k+1)}\Omega^{-(k+1)} J^{\otimes \epsilon} 
  \rightarrow 
  DA_{k, \epsilon} 
  \rightarrow 
  0
 \]
which gives a distinguished triangle:
\[
 \Omega  DA_{k, \epsilon} \rightarrow
 \field 
  \rightarrow 
  \Sigma^{(k+1)}\Omega^{-(k+1)} J^{\otimes \epsilon} \rightarrow .
\]
Forming the tensor product with $ \Sigma^{-(k+1)}\Omega^{(k+1)} J^{\otimes 
\epsilon}$ then yields the result, since the resulting distinguished triangle 
is  isomorphic to the original one, by the uniqueness of the non-trivial map in 
$[\Sigma^{-(k+1)} \Omega^{k+1} J ^{\otimes \epsilon}, \field ]$, as stated in 
Lemma \ref{lem:Picard_construction_Q0-acyclics}.

Namely, this implies that $\Sigma^{-(k+1)}\Omega^{(k+1)} J^{\otimes 
\epsilon}\otimes  
\Omega  DA_{k, \epsilon} \simeq A_{k,\epsilon}$, from which the result follows 
by 
using $A_{k,\epsilon} \otimes J \simeq \Omega^2 \Sigma^{-6} A_{k,\epsilon}$, by 
Proposition \ref{prop:Q_0-acyclic}.
\end{proof}

\begin{nota}
 For $\epsilon \in \{0,1 \}$ and $k \in \nat$, let $\orb_{k,\epsilon}$ denote 
the set of stable isomorphism classes represented by 
 $
  \{ 
 (\Sigma^{-3} \Omega)^tA_{k,\epsilon } | 0 \leq t \leq 3 
  \}.
 $
\end{nota}

Proposition \ref{prop:A_orbits_normalization} implies that, 
for $k \geq 1$,  
$\orb_{k,0} \amalg \orb _{k,1}$ 
is the set of stable isomorphism classes of $Q_0$-acyclic, finite 
$\aone$-modules with 
$Q_1$-Margolis cohomology $\field$  in degrees $\{0, 2k-1 \}$ and zero 
otherwise.

\begin{nota}
Denote by $D_8$ the dihedral group of order $8$, the non-trivial semi-direct 
product $\zed/4 \rtimes \zed/2$.
\end{nota}

\begin{cor}
\label{cor:orbit}
 For $\epsilon \in \{0,1 \}$ and $1 \leq k \in \nat$, the associations
 \begin{eqnarray*}
  A & \mapsto & \mathfrak{d}_k A := \Sigma^{2k+5} D A \\
  A & \mapsto & \Sigma^{-3} \Omega A, 
 \end{eqnarray*}
 for $A \in \orb_{k,\epsilon}$, induce an action of the group $D_8$ on 
$\orb_{k, 
\epsilon}$. The action of the centre $\zed/2 \subset D_8$ corresponds to $A 
\mapsto A \otimes J$. 

Moreover, the fixed points $\orb_{k,\epsilon}^{\mathfrak{d}_k}$ under the 
duality operator have cardinal:
\[
 |\orb_{k,\epsilon}^{\mathfrak{d}_k} | = 
 \left\{
 \begin{array}{ll}
  2 & k \equiv 0 \mod (2) \\
  0 & k \equiv 1 \mod (2) 
 \end{array}
\right.
\]
and $(\Sigma^{-3} \Omega)^tA_{k,\epsilon }$ is a $\mathfrak{d}_k$-fixed point 
if 
and only if $(\Sigma^{-3} \Omega)^tA_{k,\epsilon } \otimes J \simeq 
(\Sigma^{-3} 
\Omega)^ {t+2} A_{k,\epsilon }$ is.
\end{cor}

\begin{proof}
 By construction (and Proposition \ref{prop:Q_0-acyclic}), the group $\zed/4$ 
acts on $\orb_{k,\epsilon}$ via 
$\Sigma^{-3}\Omega$. Moreover, Proposition \ref{prop:duality_A} gives the 
stable 
isomorphism 
 \begin{eqnarray}
 \label{eqn:duality_action}
   A_{k,\epsilon} \simeq
 ( \Sigma^{-3} \Omega) ^{k+2 + 2 \epsilon} \mathfrak{d}_k A_{k, \epsilon}, 
 \end{eqnarray}
which shows that $ A  \mapsto  \mathfrak{d}_k  A $ induces an action of 
$\zed/2$ 
on  $\orb_{k,\epsilon}$. It is straightforward to verify that this defines an 
action of $D_8$ on $\orb_{k,\epsilon}$; moreover, since $J$ is self-dual, 
the action $A \mapsto A \otimes J \simeq \Sigma^{-6} \Omega^2 A$ (again 
applying 
Proposition \ref{prop:Q_0-acyclic}) corresponds to the centre. 

Equation (\ref{eqn:duality_action}) implies that 
\[
 \mathfrak{d}_k \big( (\Sigma^{-3} \Omega) ^t A_{k,\epsilon} ) \big) 
 \simeq
 (\Sigma^{-3} \Omega)^{-(t+k+2+2\epsilon)} A_{k,\epsilon},
\]
thus $(\Sigma^{-3} \Omega) ^t A_{k,\epsilon} $ is a fixed point for 
$\mathfrak{d}_k$ if and only if $2t + k+2 + 2 \epsilon  \equiv 0 \mod (4)$. 
This 
establishes 
the final statement. 
\end{proof}

\begin{exam}
\label{exam:orbits_duality}
This behaviour is illustrated by Example \ref{exam:orbits}; in the case $s=2$ 
(so that $k=1$), 
there are no fixed points under the duality operator (see Figure 
\ref{fig:orbit_s=2}) whereas there are two fixed points in the case 
$s=3$ ($k=2$) (see Figure \ref{fig:orbit_s=3}). To 
indicate that the behaviour is independent of $\epsilon$, the second orbit is 
illustrated in Figure \ref{fig:second_orbit_s=3}.

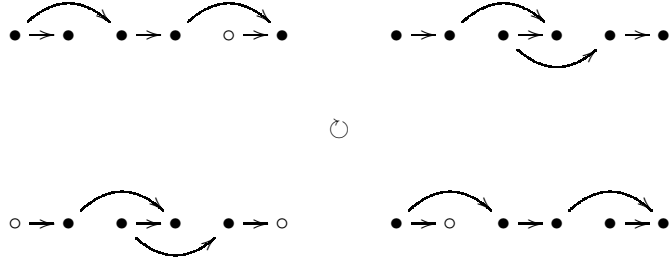
\begin{figure}
 \caption{The second orbit under $\Sigma^{-3}\Omega$ for $s=3$}
 \label{fig:second_orbit_s=3}
\[
  \xymatrix@R=.75pc @C=.75pc{
  \bullet\ar@/^1pc/[rr]
  \ar[r]
  &
  \bullet 
  &
  \bullet 
  \ar[r]
  &
  \bullet 
  \ar@/^1pc/[rr]
  &
  \circ \ar[r]
  &
  \bullet
  &&
  \bullet
  \ar[r]
  &
  \bullet 
  \ar@/^1pc/[rr]
  &
  \bullet 
  \ar[r]
   \ar@/_1pc/[rr]
  &
  \bullet
  &
  \bullet 
  \ar[r]
  &
  \bullet
  \\
  \\
  &&&&&&\circlearrowright&&&
  \\
  \\
  \circ
  \ar[r]
  &
  \bullet 
  \ar@/^1pc/[rr]
  &
  \bullet
  \ar[r]
  \ar@/_1pc/[rr]
  &
  \bullet 
  &
  \bullet 
  \ar[r]
  &
  \circ
  &
  &
  \bullet\ar@/^1pc/[rr]
  \ar[r]
  &
  \circ
  &
  \bullet
  \ar[r]
  &
  \bullet 
  \ar@/^1pc/[rr]
  &
  \bullet\ar[r]
  &
  \bullet
  }
 \]
 \end{figure}
\end{exam}

\section{Truncated projective spaces}
\label{sect:trunc_p}

This section  shows how the cohomology of certain truncated 
projective spaces fits
into the classification of Section \ref{sect:picard_interpret}. This 
leads to a conceptual understanding 
of the cohomological behaviour of these modules.

The mod $2$ cohomology $H^* (\rp^\infty) \cong \field[u]$ 
is well understood as a module over the Steenrod algebra,  with structure
entirely determined by the fact that it is an unstable algebra. Since $u^4$ is 
annihilated by $Sq^1 , Sq^2$, multiplication by $u^4$ induces a monomorphism of 
$\aone$-modules, leading to 
a form of periodicity of the $\aone$-structure. As indicated in Remark 
\ref{rem:top_real_Pn}, 
the reduced cohomology $P:= \tilde{H}^* (\rp^\infty)$ identifies as an 
$\aone$-module with $\Sigma \Omega 
P_0$.

The $\aone$-module $P$ embeds in $P_{-\infty}$, which corresponds to $\field 
[u^{\pm 1}]$; this works over the full Steenrod algebra $\cala$ but is simpler 
over $\aone$, since this corresponds to inverting $u^4$ and imposing 
periodicity.

\begin{nota}
 For $a \leq b \in \zed$, let $P^b_a$ denote the subquotient of $P_{-\infty}$ 
of 
elements 
in degrees $[a, b]$. 
\end{nota}

\begin{rem}
\ 
\begin{enumerate}
 \item 
 The $\aone$-modules $P^b_a$ can be realized as the mod $2$ cohomology of $\rp^b 
_a$, namely, for $a>0$, truncated projective space and, for $a \leq 0$, the 
appropriate Thom spectrum.
\item 
Periodicity as $\aone$-modules gives  
$
 \Sigma^4 P^{b}_{a} 
 \cong 
 P^{b+4}_{a+4} 
$
hence, up to suspension, it is sufficient to consider the cohomology of 
truncated projective spaces.
 \end{enumerate}
\end{rem}

Duality in this context is straightforward:

\begin{lem}
 \label{lem:duality_truncP}
For $a \leq b \in \zed$, $P^b_a$ there is an isomorphism:
\[
 DP^b_a \cong \Sigma^{1-4b}P^{4b-a-1}_{3b-1}. 
 \]
\end{lem}

\begin{proof}
It is a basic fact (see \cite{bruner_Ossa} for example) that the dual of 
$P_{-\infty}$ is $\Sigma P_{-\infty}$. The dual of $P^b_a$ is 
$(DP_{-\infty})^{-a}_{-b}$ which is isomorphic to 
$(\Sigma P_{-\infty})^{-a}_{-b} = \Sigma \big( (P_{-\infty})^{-a-1}_{-b-1} 
\big)$. Using periodicity, shifting by $4b$ gives the statement.  
\end{proof}

Here only the $Q_0$-acyclic cases are considered, namely those of the form 
$P^{2n} _{2m-1}$, with $1 \leq m \leq n$. The case $m-n=0$ is easily 
understood, 
so 
henceforth suppose that $m-n >0$.

\begin{thm}
\label{thm:identify_A_trunc}
 For natural numbers $1 \leq m < n$, 
 \begin{enumerate}
  \item 
  if $m \equiv 1 \mod (2)$, 
  \[
   P^{2n}_{2m-1} \cong \left\{ 
   \begin{array}{ll}
    \Sigma^{2m-3} A_{n-m, 0} & n-m \equiv 2, 3 \mod (4) \\
    \Sigma^{2m-3} A_{n-m, 1} & n-m \equiv 0, 1 \mod (4); \\
   \end{array}
   \right.
  \]
\item 
if $m \equiv 0 \mod (2)$ 
\[
   P^{2n}_{2m-1} \cong\left\{ 
   \begin{array}{ll}
    \Sigma^{2m-3 }  (\Omega^{-1} \Sigma^3)    A_{n-m, 0} & n-m \equiv 0, 3 \mod 
(4) \\
    \Sigma^{2m-3 } (\Omega^{-1} \Sigma^3)  A_{n-m, 1} & n-m \equiv  1, 2 \mod 
(4);\\
   \end{array}
   \right.
  \]
\end{enumerate}
\end{thm}

\begin{proof}
All the modules considered are reduced, hence it suffices to establish that 
there are stable isomorphisms of the above form.

Since the $A_{k, \epsilon}$ generate the respective orbits (see Proposition 
\ref{prop:A_orbits_normalization}), there exists a stable isomorphism of the 
form:
\[
 P^{2n}_{2m-1} \simeq \Sigma ^{2m-3} (\Omega \Sigma^{-3} ) ^{t(m,n)} A_{n-m, 
\epsilon(m,n)},
\]
where the suspension and the value $k=m-n$ are determined by the $Q_1$-Margolis 
cohomology groups. It remains to determine $t$ and $\epsilon$ as functions of 
$m,n$. 

In the case $m \equiv 1 \mod (2)$, by inspection $Sq^2$ acts trivially on the 
bottom class and there is an inclusion of the question mark complex (suitably 
suspended) into 
$P^{2n}_{2m-1}$ (compare Section \ref{sect:basic}). It follows that $t=0$, by 
the 
characterization of $A_{k, \epsilon}$ given in Proposition 
\ref{prop:A_orbits_normalization}. 
To conclude in this case, it suffices to select $\epsilon$ so that $A_{n-m, 
\epsilon}$ does not have a Joker at the top; the choice is dictated by 
Proposition \ref{prop:identify_A} 
(see also Figures \ref{fig:Ak0} and \ref{fig:Ak1}). 

If  $m \equiv 0 \mod (2)$, consider $\Omega \Sigma^{-3}P^{2n}_{2m-1}$;  
 there is once again an inclusion of the question mark complex (suitably 
suspended), so the normalization of Proposition 
\ref{prop:A_orbits_normalization} now gives:
\[
 P^{2n}_{2m-1} \simeq \Sigma^{2m-3} (\Omega^{-1} \Sigma^{3}) A_{n-m, \epsilon 
(m,n)}.
\]
The argument proceeds as above, but taking into account the cyclic shift 
induced 
by 
$(\Omega^{-1} \Sigma^{3})$.
\end{proof}

\begin{rem}
 This result explains the intricacy of some of the statements concerning the 
modules $  P^{2n}_{2m-1}$ in \cite{davis}. It is much simpler
 to work with the modules $A_{k,\epsilon}$.
\end{rem}

\section{Dual Brown-Gitler modules}
\label{sect:brown-gitler} 

This section gives a description of the dual Brown-Gitler modules (more 
precisely, 
the associated reduced $\aone$-modules) in terms of the families defined in 
Section  \ref{sect:picard_interpret}. Here the 
conclusion is  simpler than for truncated projective spaces, although the 
structure of the modules is richer (see Theorem \ref{thm:identify_BG}).

\begin{rem}
\label{rem:BG_spectra}
With topological applications in mind and to fix conventions,  recall the 
families $B(k), B_0 (k)$ indexed by $\nat$, where $B(k)$ is the $k$th 
Brown-Gitler spectrum \cite{BG,GLM} and $B_0 (k)$ is used to denote the $k$th 
integral 
Brown-Gitler spectrum (cf. \cite{Shimamoto}, where a different indexing is
used); the notation $B_0$ follows \cite{pearson}. 
The indexing below follows that of \cite{GLM}, in particular there are homotopy 
equivalences
$B(2k) \simeq B(2k+1) $ and $B_0 (2k)\simeq B_0 (2k+1)$, so consideration is 
limited to the even-indexed spectra.  

There is a map $i_k : B_0 (k) \rightarrow  B(k)$ which induces the 
surjection on mod $2$ cohomology 
\[
 H^* (B(k) ) \cong \cala/ \cala \{\chi (Sq^i ) | 2i >k \}
\twoheadrightarrow
 H^* (B_0(k) ) \cong \cala/ \cala \{Sq^1, \chi (Sq^i ) | 2i >k \},
\]
where $\chi$ is the conjugation of $\cala$. 

The relationship between these families is made clearer by the following facts 
for $k \in \nat$, : 
\begin{enumerate}
 \item 
$B_0 (4k+2 ) \simeq B_0 (4k)$;
\item 
$i_{4k+2}$ induces an equivalence $$(S^0/2)\wedge B_0 (4k) \simeq (S^0/2)\wedge 
B_0 (4k+2) \simeq B(4k+2);$$ 
\item 
there is a cofibre sequence (see \cite{Shimamoto}) 
\begin{eqnarray*}
 \label{eqn:shim_cofib}
 B_0 (4k+ 4) \stackrel{i_{4k+4}} {\rightarrow} B(4k+4) \rightarrow \Sigma B_0 
(4k) 
\end{eqnarray*}
(this choice of triangle is for later convenience).
\end{enumerate}
Here the dual Brown-Gitler spectra $D B(k)$ and $DB_0 (k)$ are of greater 
relevance, where $D$ denotes Spanier-Whitehead duality. 
The mod $2$ cohomology of $DB(k)$ is related to the family of injective  
Brown-Gitler modules in unstable module theory \cite{sch} by 
\[
 H^* (DB(k)) \cong \Sigma^{-k} J(k)
\]
where $J(k)$ is the injective envelope in unstable modules of $\Sigma ^k 
\field$. In particular, the top dimensional cohomology class of $DB(k)$ is in
degree zero.
 The families of modules considered below correspond respectively to $H^* 
(DB(k))$ and $H^* (DB_0 (k))$. 
\end{rem}

Recall that the dual  Steenrod algebra $\cala^*$ is a polynomial algebra 
$\field[\xi_i| i \geq 1]$ on generators of cohomological degree $|\xi_i|= 
1 - 2^i$; the generators can be replaced 
by their conjugates $\zeta_i := \chi (\xi_i)$. 

The natural action of $\cala$ on its dual $\cala^*$  is on the right; since 
$\cala$ is a 
Hopf algebra (in particular, with conjugation) 
the category of right $\cala$-modules is equivalent to that of left 
$\cala$-modules. As noted in  Remark \ref{rem:left_right_A1}, 
over $\aone$ the identity  $\chi (Sq^i) = Sq^i $ for $i \in \{1, 2 \}$  
renders 
the translation simple. 

With respect to the algebra generators $\zeta_i$,  
the above action (on the right) is given by the 
action of the Steenrod total power 
\[
 \zeta_n Sq = \sum _{i=0}^n \zeta_{n-i}^{2^i},
\]
where $\zeta_0$ here is interpreted as $1$.

For current purposes,  $\zeta_0$ is considered as an independent 
generator of degree zero (with $\zeta_0 Sq = \zeta_0$) and equip the generators 
$\zeta_i$ (for $i \geq 0$) with weights:
\[
 \wt (\zeta_i) =2^i
\]
so that the action of the Steenrod algebra preserves the weights.

In  particular, considering the left $\aone$-action,  
for $n >0$:
\begin{eqnarray*}
Sq^1  \zeta_n &=& \zeta_{n-1}^2 \\
Sq^2 \zeta_n &=& 0 \\
Sq^2 (\zeta_n)^2  &=& \zeta_{n-1}^4.
\end{eqnarray*}

The sub-algebras:
\[
\wfour:=  \field [\zeta_0^4, \zeta_1^2, \zeta_2 , \ldots ] 
 \hookrightarrow 
\wtwo :=  \field [\zeta_0^2, \zeta_1, \zeta_2 , \ldots ] 
\hookrightarrow 
 \field [\zeta_0, \zeta_1, \zeta_2 , \ldots ] 
 \]
are stable under the action of the Steenrod algebra. More precisely one 
has:

\begin{lem}
 \label{lem:weights}
 There are weight decompositions in the category of  $\cala$-modules:
 \begin{eqnarray*}
  \wfour &\cong & \bigoplus _{i\geq 0} \wfour (4i) \\
   \wtwo &\cong & \bigoplus _{j\geq 0} \wtwo (2j) 
 \end{eqnarray*}
where $\wfour (n)$ (respectively $\wtwo (n)$) is the subspace of $\wfour$ 
(resp. $\wtwo$) of  elements of weight $n$.  
 Moreover, each $\wfour (n)$ (respectively $\wtwo(n)$) is of finite total 
dimension.
\end{lem}

\begin{rem}
\label{rem:identify_DB}
There are isomorphisms of $\aone$-modules:
\begin{eqnarray*}
\wtwo (2n) &\cong& H^* (DB (2n)) \\
\wfour (4n) &\cong& H^* (DB_0 (4n)).
\end{eqnarray*}
The notation reflects the fact that the spectrum $\Sigma^k DB(k)$ is usually 
written $T(k)$, so that $H^* (T(k)) \cong J(k)$. 
Thus $\wtwo (2n) \cong \Sigma^{-2n} J (2n)$. 
\end{rem}

\begin{exam}
 \label{exam:Z}
 There is an isomorphism $\wtwo(2)\cong Z$, with basis $\{\zeta_1, \zeta_0^2 
\}$. 
\end{exam}

Using Remark \ref{rem:identify_DB}, the following Lemma gives algebraic
versions 
of the 
Shimamoto exact sequences of Remark \ref{rem:BG_spectra}:

\begin{lem}
 \label{lem:wtfour-module}
 As an $\wfour$-module, $\wtwo$ is free on $\{1, \zeta^2_0, \zeta_1 , \zeta^2_0 
\zeta_1 \}$, which are of weights $0, 2, 2, 4$ respectively.
 In particular, for $n \in \nat$:
 \begin{eqnarray*}
  \wtwo(4n+2) \simeq  \wfour(4n) \otimes  \wtwo(2) 
 \end{eqnarray*}
as  $\cala$-modules and there is a short exact sequence of $\cala$-modules:
\[
 0 
 \rightarrow 
 \wfour(4n+4) 
 \rightarrow 
 \wtwo(4n+4) 
 \rightarrow 
 \Sigma^{-1}  \wfour(4n) 
 \rightarrow 
 0.
\]
\end{lem}

\begin{proof}
 Straightforward (the desuspension in the short exact sequence arises since 
$\zeta^2_0 \zeta_1$ has  degree $-1$).
\end{proof}

The aim here is to understand the structure of the  $\wfour(n)$ 
(respectively 
$\wtwo(n)$) as $\aone$-modules;   a first step is understanding the action 
of 
the Milnor primitives. 
The following is clear (cf. \cite{ravenel}, for example). 

\begin{lem}
\label{lem:actionQ_zeta}
 For $i, n \in \nat$
 \[
 Q_i \zeta_n  = \left\{ 
  \begin{array}{ll}
 0 & i \geq n \\
 \zeta_{n-(i+1)}^{2^{i+1}} & i< n.
  \end{array}
\right.
 \]
\end{lem}

Since the Milnor operations acts as derivations, this allows the calculation of 
the Margolis cohomology groups. For $n \in \nat$, write $\alpha (n)$ for the
sum 
of the digits of its 
 binary expansion.

\begin{lem}
\label{lem:Margolis_cohom_calc}
 (Cf. \cite[Part III]{adams}, \cite[Lemma 3.12]{DGM}.)
 There are isomorphisms of algebras:
\begin{eqnarray*}
 H^* (\wtwo , Q_0 ) & \cong & \field \\
 H^* (\wtwo , Q_1) & \cong & \field[\zeta_1]/(\zeta_1^4) \otimes 
\bigotimes_{i\neq 1} \field [\zeta_i^2]/ (\zeta_i)^4 \\
 H^* (\wfour , Q_0 ) & \cong & \field [\zeta_0^4] \\
 H^* (\wfour , Q_1) & \cong & \bigotimes_{i\geq 1} \field [\zeta_i^2]/ 
(\zeta_i)^4.
\end{eqnarray*}
In particular, for $n \in \nat$,
 $
 \dim H^* (\wfour(4n) , Q_0) = \dim H^* (\wfour(4n) , Q_1) = 1,
$ 
so that the modules $\wfour(4n)$ represent elements of the 
Picard group $\pica$, with $H^{0}(\wfour(4n), Q_0) = \field$ and $H^* 
(\wfour(4n) , Q_1)$ concentrated in degree 
$$\sum_{j \in \mathscr{J}} 2 (1- 2^{i_j+1} )= 2( \alpha (n)-2n),$$
where $n= \sum_{j \in \mathscr{J}}2^{i_j}$ is the dyadic decomposition.

The modules $\wtwo(2n)$ are $Q_0$-acyclic for $n >0$ and there are 
isomorphisms:
\begin{eqnarray*}
 H^*(  \wtwo(4n+4), Q_1 ) &\cong&  H^* ( \wfour(4n+4), Q_1) \oplus 
\Sigma^{-1} H^* ( \wfour(4n) , Q_1)\\
H^*(  \wtwo(4n+2), Q_1 ) &\cong& H^*(  \wfour(4n), Q_1 ) \otimes 
\wtwo(2).
\end{eqnarray*}
In particular, the modules $\wtwo(2n)$, for $n>0$ are stably isomorphic to 
$Q_0$-acyclic, indecomposable $\aone$-modules. 
\end{lem}

\begin{proof}
Straightforward.
\end{proof}

The result of \cite[Lemma 3.12]{DGM} is more precise: the  modules 
$\wfour(4.2^i)$, $ i \in \nat$, are identified explicitly in $\pica$ 
and it is observed that the multiplicative structure of 
$\wfour$ induces stable isomorphisms
\[
 \bigotimes_{j \in \mathscr{J}} \wfour(4.2^{i_j})
 \simeq 
 \wfour(4n),
\]
where $n= \sum_{j \in \mathscr{J}}2^{i_j}$ is the dyadic decomposition. The 
result can be restated as follows (cf. Example \ref{exam:Q1_picard_family}):

\begin{lem}
\label{lem:wfour_picard_gp}
 For $i\in \nat$, there is a stable isomorphism of $\aone$-modules:
 \[
  \wfour(4.2^{i})
  \simeq 
  \left\{
  \begin{array}{ll}
  \Sigma \Omega^{-1} J & i=0 \\
  (\Sigma \Omega^{-1}) ^{2^{i+1}-1 } \field & i>0 .
  \end{array}
  \right.
\]
Hence, for $0< n \in \nat$, 
 $
 \wfour(4n)
  \simeq 
  (\Sigma \Omega^{-1}) ^{2 n- \alpha (n) }
 \otimes J ^{\otimes \overline{n} }, 
$ where $\overline{n}$ is the residue of $n$ modulo $2$.
\end{lem}

A similar argument applies to the modules $\wtwo(2n)$:

\begin{lem}
\label{lem:wtwo_reduction}
 For $0<n \in \nat$ with $2$-adic valuation $\nu = \nu (n)$,  the 
structure 
of $\wtwo$ as a $\wfour$-module induces a stable isomorphism 
\begin{eqnarray*}
  \wtwo(2n)
  &\simeq& 
  \wfour(2n - 2^{\nu +1}) \otimes \wtwo(2^{\nu +1}).
 \end{eqnarray*}

Hence 
$
 \Sigma^{2n} \wtwo (2n ) \simeq  (\Sigma \Omega^{-1} ) ^{1- \alpha (n)} \big( 
\Sigma^{2^{\nu +1}} \wtwo (2^{\nu +1})\big).
$
\end{lem}

\begin{proof}
For notational simplicity, write  $n= 2 m + 2^{\nu}$ (so that $m = 2^{\nu} t$
for some $t \in \nat$); observe 
that $\alpha (m) = \alpha (n)-1$. 

In the case $\nu =0$, the stable isomorphism $ \wtwo(2n)
  \simeq   \wfour(4m) \otimes \wtwo(2^{\nu +1}) $ follows immediately from
Lemma 
\ref{lem:wtfour-module}, hence consider the case $\nu >0$ (so that $2n \equiv 0 
\mod (4)$).
Both sides of the expression are $Q_0$-acyclic and Lemma 
\ref{lem:Margolis_cohom_calc} shows that they have isomorphic $Q_1$-Margolis 
cohomology groups, hence it remains to check that the multiplication induces an 
isomorphism in $H^* (-,Q_1)$; this can be seen more explicitly as follows. 
 
The multiplication together with the short exact sequences given by 
Lemma \ref{lem:wtfour-module} yield a commutative diagram
\[
 \xymatrix{
\wfour (4m) \otimes \wfour (2^{\nu +1}) 
\ar[r]
\ar[d]
&
\wfour (4m) \otimes \wtwo (2^{\nu +1}) 
\ar[r]
\ar[d]
&
\wfour (4m) \otimes \Sigma^{-1}\wfour (2^{\nu +1} -4) 
\ar[d]
\\
\wfour (2n)
\ar[r]
&
 \wtwo (2n) 
\ar[r]
&
\Sigma^{-1}\wfour (2n-4) 
}
\]
in which the rows are short exact. The outer vertical morphisms are stable 
isomorphisms by \cite[Lemma 3.12]{DGM}, hence so is the middle one. 

Lemma \ref{lem:wfour_picard_gp} identifies  $\wfour (4m)$ up to stable 
isomorphism, which gives:
\begin{eqnarray*}
  \wtwo(2n)
  &\simeq& 
  (\Sigma \Omega^{-1}) ^{2 m- \alpha (m) } \wtwo(2^{\nu +1}) \otimes 
J^{\otimes \overline{m}}
 \end{eqnarray*}
where $\overline{m}\in \{0,1 \}$ is the residue of $m$ modulo $2$. 

The factor $J^{\otimes \overline{m}}$ can be removed using the fact that 
$\wtwo(2^{\nu +1 
})$ 
is 
$Q_0$-acyclic, hence Proposition \ref{prop:Q_0-acyclic} implies that 
$
  \wtwo(2^{\nu +1}) \otimes J^{\otimes \overline{m}}\simeq 
  \Sigma^{-6 \overline{m}} \Omega^{2 \overline{m} } \wtwo(2^{\nu +1 }).
$
It follows that 
\begin{eqnarray*}
  \wtwo(2n)
  &\simeq& 
 \Sigma^{-4 \overline{m}}  (\Sigma \Omega^{-1}) ^{2 ( m- \overline{m}) - 
\alpha 
(m) } \wtwo(2^{\nu +1}).
 \end{eqnarray*}

The result is then rewritten in terms of $\Sigma^k \wtwo (k)$ for $k \in \{ 2n, 
2^{\nu +1} \}$, 
which gives 
\[
 \Sigma^{2n} \wtwo (2n ) \simeq (\Sigma^3 \Omega^{-1} ) ^{2(m-\overline{m})} 
(\Sigma \Omega^{-1} ) ^{- \alpha (m)} \big( \Sigma^{2^{\nu +1}} \wtwo (2^{\nu 
+1}) \big).
\]
Now, $m- \overline{m} \equiv 0 \mod (2)$, so $2(m-\overline{m})\equiv 0 \mod 
(4)$. Thus  Corollary \ref{cor:orbit} implies 
that 
$$(\Sigma^3 \Omega^{-1} ) ^{2(m-\overline{m})}\wtwo (2^{\nu +1}) \simeq \wtwo 
(2^{\nu +1});
$$ this yields the final statement.
\end{proof}

It remains to identify the stable isomorphism classes of the 
$\wtwo(2^{\nu +1})$. The case $\nu=0$ is already known (see Example 
\ref{exam:Z}). The general case is addressed using the algebraic Mahowald 
exact sequences \cite{mahow_new_inf}.

\begin{lem}
\label{lem:mahowald_ses}
For $n \in \nat$, there is a short exact sequence of $\aone$-modules:
\[
 0
 \rightarrow 
 \wtwo(2(2n -1))
\rightarrow
\wtwo(4n)
\rightarrow
\Sigma^{-2n}
\wtwo(2n)
\rightarrow 
0.
\]
\end{lem}

\begin{proof}
The submodule $ \wtwo(2(2n -1)) $ corresponds to those 
terms which are divisible by $\zeta^2_0$. The quotient is identified by  
re-indexing by $\zeta_{i} \mapsto \zeta_{i-1}$ (which divides weights by two); 
 this gives the suspension, since the highest dimensional term in the quotient 
is  $\zeta_1^{2n}$ (which has 
degree $-2n$), which re-indexes to $\zeta_0^{2n}$ (of degree zero).  
\end{proof}

This is applied in the case $n=2^{\nu -1}$ (for $\nu >0$):

\begin{lem}
\label{lem:mahowald_special_case}
 For $0< \nu \in \nat$, there is  a short exact sequence 
\[
  0
 \rightarrow 
 \wtwo(4(2^{\nu -1}-1) +2) 
\rightarrow
\wtwo(2^{\nu +1})
\rightarrow
\Sigma^{-2^{\nu}}
\wtwo(2^{\nu})
\rightarrow 
0.
\]
Moreover, there is  a stable isomorphism
$
 \wtwo(4(2^{\nu -1}-1) +2)
\simeq 
\Sigma^{5-\nu -2^{\nu +1} } \Omega^{\nu-1} Z.
 $

For $\nu=1$, this gives a non-trivial short exact sequence:
\[
  0
 \rightarrow 
Z
\rightarrow
\wtwo(4)
\rightarrow
\Sigma^{-2}Z
\rightarrow 
0.
\]
\end{lem}

\begin{proof}
 The short exact sequences are provided by Lemma \ref{lem:mahowald_ses}.

By Lemma \ref{lem:wtwo_reduction}, there is a stable isomorphism 
$$
\Sigma^{2^{\nu +1} -2} \wtwo(4(2^{\nu -1}-1) +2)
\simeq 
(\Sigma \Omega^{-1}) ^{1 - \nu } (\Sigma^2 \wtwo (2)),
$$
since $\alpha (2^{\nu -1} -1) = \nu -1$. Since $\wtwo (2) \simeq Z$, rewriting 
provides the stated stable isomorphism.
\end{proof}

\begin{lem}
\label{lem:wtwo_inject_question}
 For $0 < \nu  \in \nat$, there is a unique non-trivial morphism
 \[
  \Sigma^{5 -2^{\nu +1}}\Omega^{-1} J  \rightarrow \wtwo(2^{\nu +1})
 \]
and this is a monomorphism. 
\end{lem}

\begin{proof}
The inclusion of the desuspension of the question-mark complex is given by 
inspection of the low-dimensional structure of $\wtwo(2^{\nu +1})$, which is 
represented by:
\[
 \xymatrix @C=1pc @R=1pc{
\zeta_{\nu -1}^4
\\
\zeta_{\nu} \zeta_{\nu -1} ^2 
\ar[u]
\\
\zeta_{\nu}^2 \ar@/^2pc/[uu]
\\
\zeta_{\nu +1} \ar[u]
 }
\]
where the element $\zeta_{\nu}^2$ of degree $2 (1-2^{\nu})$ represents the 
lowest dimensional $Q_1$ Margolis cohomology class. 
\end{proof}

\begin{rem}
 For $\nu=1$, the diagram in the proof of Lemma \ref{lem:wtwo_inject_question} 
describes the structure of $\wtwo(4)$, which corresponds (up to suspensions) 
with the reduced cohomology of $\mathbb{R}P^4$.
\end{rem}

\begin{thm}
 \label{thm:identify_BG}
For $0<n \in \nat$ with $2$-adic valuation $\nu := \nu (n)$, there is a stable 
isomorphism
\[
 \Sigma^{2n} \wtwo(2n)
\simeq (
\Sigma \Omega^{-1} ) ^{1 - \alpha (n)} \big( \Sigma^{-1} A_{\nu, 1} \big),
\]
where one takes $A_{0,1}:=\Sigma^3 Z$.
 \end{thm}

\begin{proof}
By Lemma \ref{lem:wtwo_reduction}, it suffices to show that 
$\Sigma^{2^{\nu+1}}\wtwo(2^{\nu+1}) \simeq \Sigma^{-1} A_{\nu,1} $, for $\nu 
\in 
\nat$.

The case $\nu =0$ follows by the identification $\wtwo (2) \simeq Z$, hence 
suppose that $\nu >0$.
Lemma \ref{lem:wtwo_inject_question} provides an embedding 
$
  \Sigma^{5 -2^{\nu +1}}\Omega^{-1} J  \rightarrow \wtwo(2^{\nu +1}).
 $
 Thus, by the normalization result of Proposition 
\ref{prop:A_orbits_normalization}, one has 
\[
 \wtwo(2^{\nu+1}) \simeq  \Sigma^{-(1+2^{\nu+1})} A_{k,\epsilon} 
\]
for some $\epsilon \in \{0, 1 \}$ and $k \in \nat$. The degrees of the 
$Q_1$-Margolis cohomology classes of $\wtwo (2^{\nu +1})$ imply that $k = \nu$. 

To show that $\epsilon=1$, one uses the Mahowald  extension of Lemma 
\ref{lem:mahowald_special_case}. By considering the behaviour of the underlying 
$\eone$-modules, the monomorphism
\[
 \Sigma^{5-\nu -2^{\nu +1}} \Omega^{\nu -1 } Z \hookrightarrow 
\wtwo(2^{\nu+1})
\]
 must correspond to a non-trivial class in 
\begin{eqnarray*}
 [\Sigma^{5-\nu -2^{\nu +1}} \Omega^{\nu -1 } Z , \Sigma^{-(1+2^{\nu+1})} 
A_{\nu,\epsilon} ]&=& 
[\Sigma^{5} (\Sigma^{-1}\Omega)^{\nu -1 } Z , A_{\nu,\epsilon} ]\\
&=& [Z, \Sigma^{-5} (\Sigma^{-1}\Omega)^{1- \nu }A_{\nu,\epsilon} ].
\end{eqnarray*}
Now, by Proposition \ref{prop:identify_A}, $A_{\nu,\epsilon}$ is a submodule of 
$\Sigma^{4 \epsilon} 
(\Omega \Sigma^{-1})^{\nu +1 - 2 \epsilon} DP_0$. Moreover, as in Lemma 
\ref{lem:question_stable}, 
one can replace stable morphisms by $\hom_{\aone}$:
\begin{eqnarray*}
 [Z, \Sigma^{-5} (\Sigma^{-1}\Omega)^{1- \nu }A_{\nu,\epsilon} ]&\cong 
&\hom_{\aone} (Z, \Sigma^{-5} (\Sigma^{-1}\Omega)^{1- \nu }A_{\nu,\epsilon} )
\\
&\cong& 
\hom_{\aone} (Z, \Sigma^{4 \epsilon -5} (\Sigma^{-1}\Omega)^{2-2\epsilon }DP_0 
),
\end{eqnarray*}
where the second isomorphism is proved by a simple connectivity argument that 
shows that the truncation in Proposition \ref{prop:identify_A} does not affect 
the 
value. 

Finally, inspection shows that the latter group is non-trivial if $\epsilon =1$ 
but zero if $\epsilon =0$ (in which case, the presence of a Joker provides the 
obstruction).
\end{proof}

\section{Ext calculations}
\label{sect:ext}

This section exploits the presentation of the $\aone$-modules $A_{k,\epsilon}$
in terms of the elements of the Picard group $\pica$ to give a uniform approach
to
$\ext_{\aone}$-calculations 
for these families; these become 
particularly transparent when graded over $\pica$. The main result is Theorem
\ref{thm:calculate_stext_Ake},
 which encompasses earlier results of Davis \cite{davis} and Pearson
\cite{pearson}. For an explanation of the 
 ext charts which illustrate these results, see Convention \ref{conv:ext_chart}.

 \begin{rem}
These calculations have immediate applications via the Adams spectral
sequence, as in the work of Davis and Pearson cited above.
 \end{rem}

It is useful to work with the stable form of $\ext$ defined using the stable 
$\aone$-module category and to grade by the 
Picard group $\pica \cong \zed^{\oplus 2} \oplus \zed/2$ \cite{adams_priddy} 
(cf. Section \ref{sect:applications}), taking   $(s, t, \epsilon) \in 
\zed^{\oplus 2} \oplus \zed/2$ to correspond to 
$ \Omega^{-s} \Sigma^t J ^{\otimes \epsilon} \in \pica$. 

\begin{defn}
 For $\aone$-modules $M, N$ and$(s,t,\epsilon) \in \pica$, $$ \stext
^{s,t,\epsilon} (M,
N) 
:= [ M , \Omega^{-s} \Sigma^t  J^{\otimes \epsilon} N].
$$
By convention, $\epsilon$ is taken to be zero where only a bigrading $(s,t)$ is
specified. 
\end{defn}

\begin{nota}
  For $\aone$-modules $M, N$, write $\stext (M,N)$ for the $\pica$-graded
stable 
ext groups.
\end{nota}

Similarly, $\ext$ in the category of $\aone$-modules is  bigraded  
$
 \ext^{s,t} _{\aone} (M,N) = \ext^s _{\aone} (M, \Sigma^t N), 
$
for $s \in \nat $ and $t \in \zed$. There is a natural morphism 
\[
 \ext^{s,t} _{\aone} (M,N) \rightarrow \stext ^{s,t} (M,N)
\]
which is an isomorphism for $s>0$ and is surjective for $s=0$; Lemma 
\ref{lem:stable_hom_reduced} gives a criterion for an isomorphism in degree 
zero. Whereas $\ext^{s,t}_{\aone} (M,N)$ is trivial for $s<0$, $\stext  ^{s,t}
(M,N)$ 
is highly non-trivial in general.

\begin{rem}
With topological applications via the Adams 
spectral sequence in view, Adams indexing $(t- 
s, s )$ is frequently used. Observe that $\stext^{s,t} (M,N) \cong 
[(\Sigma^{-1} \Omega )^s M, 
\Sigma^{t-s} N]$, so it can be useful to 
use $\Sigma^{-1} \Omega $ in place of $\Omega$.
\end{rem}

There is an exterior product:
\[
 \stext(\field , M_1) 
 \otimes 
 \stext (\field , M_2) 
\rightarrow
 \stext  (\field , M_1 \otimes M_2) 
 \]
induced by the tensor structure of the stable module category. In particular 
$\stext (\field, \field) $ has the structure of 
a commutative, $\pica$-graded $\field$-algebra and, for $M$ an $\aone$-module, 
$\stext (\field, M)$ is naturally a $\pica$-graded module over this algebra.

\begin{lem}
\label{lem:ext_trigraded}
(Cf. \cite[Lemma 3.3]{davis} and \cite{bg2}, for example).
 The sub $\pica$-graded algebra of  $\stext  
(\field, \field)$ of classes with $t-s \geq 0$ is isomorphic to 
 \[
  \field [h_0, h_1, \kappa, \alpha] / (h_0h_1, h_1^3, h_1\kappa  , 
\kappa^2 -  h_0^2),
 \]
where, with respect to Adams indexing $(t-s, s, \epsilon)$:
\begin{eqnarray*}
 |h_0|= (0,1,0), \ 
  |h_1|=(1,1,0), \ 
  |\kappa|=(0,1,1),\
  |\alpha|= (4,2,1).
\end{eqnarray*}
For $a, b$  the generators of Adams indexing $|a|= (4,3,0)$ and $|b|= 
(8,4,0)$, one has:
\begin{eqnarray*}
 a &=& \kappa \alpha\\
 b&=& \alpha ^2. 
\end{eqnarray*}
Hence there is an isomorphism of $\pica$-graded algebras
\[
 \stext (\field , \field)[\alpha^{-1}]
 \cong 
  \field [h_0, h_1, \kappa, \alpha^{\pm 1}] / (h_0h_1, h_1^3, h_1\kappa  , 
\kappa^2 - h_0^2).
\]
\end{lem}

\begin{rem}
\ 
\begin{enumerate}
\item
The class $b$ is the usual algebraic Bott  class. The relations $b= \alpha^2$, 
$a=\kappa \alpha$ and 
$\kappa^2 = h_0^2$ imply the standard relation $a^2 = h_0^2 b$. Inverting
$\alpha$ is equivalent to inverting $b$.
\item 
There is a non-trivial  class in $\stext^{*,*} (\field, J)$ 
of Adams index $(t-s, s)=(-2,0)$; multiplication by  $b$ acts injectively on 
this class. Hence the subalgebra of Lemma \ref{lem:ext_trigraded} of elements 
with $t-s \geq 0$ does not contain $\ext^{s,t}(\field, J)$.
\item 
For $t-s \geq 0$, the above exhausts $\stext$; however, as indicated in 
Figures  \ref{fig:A21} and \ref{fig:A20}, the third quadrant corresponds to the 
first by rotational symmetry about $(-2.5,-.5)$ hence is highly non-trivial. 
(See Example \ref{exam:Z_periodic} below for indications on how to derive this
symmetry.)
\end{enumerate}
\end{rem}

\begin{nota}
 Write $\e$ for the $\pica$-graded algebra
 \[
  \e := \stext (\field , \field)[b^{-1}], 
 \]
 corresponding to the Bott-periodic version of stable ext.
\end{nota}

\begin{conv}
\label{conv:ext_chart}
As usual, $\stext$ charts are indicated in the  $(t-s,s)$-plane and 
multiplication by $h_0$ is indicated by a vertical line and by $h_1$ by a 
diagonal line, where $h_0$, $h_1$ are the usual generators in $\ext^{1,*}_\aone 
(\field, \field)$. 

Where the chart is $b$-periodic, only a portion of the chart which generates
under $b$-periodicity is given; 
in the remaining cases considered here, the charts can be completed in the left
and right hand half planes 
by understanding the action of $b$.
\end{conv}

The chart for $\stext^{*,*} (\field, \field)$ is given in Figure 
\ref{fig:stext_F} and for $\stext^{*,*} (\field,J)$ in Figure \ref{fig:stext_J}.

\begin{figure}
 \caption{The chart for $\stext^{*,*} (\field,\field)$}
 \label{fig:stext_F}
\begin{picture}(15,13)(-10,-6)
\thinlines
\put(-5,-1){\circle*{0.2}}
\put(-6,-.7){$\scriptstyle{(-5,-1)}$}
\put(-5,-1){\line(0,-1){4}}
\put(-5,-1){\line(-1,-1){2}}
\put(-5,-2){\circle*{0.2}}
\put(-5,-3){\circle*{0.2}}
\put(-5,-4){\circle*{0.2}}
\put(-6,-2){\circle*{0.2}}
\put(-7,-3){\circle*{0.2}}
\put(-9,-4){\circle*{0.2}}
\put(-9,-4){\line(0,-1){1}}
\put(0,0){\circle*{0.2}}
\put(0,-.5){$\scriptstyle{(0,0)}$}
  \put(0,1){\circle*{0.2}}
  \put(0,2){\circle*{0.2}}
   \put(0,3){\circle*{0.2}}
    \put(0,4){\circle*{0.2}}
 \put(1,1){\circle*{0.2}}
  \put(2,2){\circle*{0.2}}
   \put(4,5){\circle*{0.2}}
   \put(4,4){\circle*{0.2}}
   \put(4,3){\circle*{0.2}} 
   \put(0,0){\line(0,1){6}}
   \put(0,0){\line(1,1){2}}
   \put(4,3){\line(0,1){3}}
   \put(3.7,2.5){$a$}
   \put(3.5,2){$\scriptstyle{(4,3)}$}
\end{picture}
\end{figure}

\begin{figure}
 \caption{The chart for $\stext^{*,*} (\field,J)$}
 \label{fig:stext_J}
\begin{picture}(15,13)(-10,-7)
\thinlines
\put(-5,-2){\circle*{0.2}}
\put(-6.5,-1.5){$\scriptstyle{(-5,-2)}$}
\put(-5,-2){\line(0,-1){4}}
\put(-5,-3){\circle*{0.2}}
\put(-5,-4){\circle*{0.2}}
\put(-5,-5){\circle*{0.2}}
\put(-9,-3){\circle*{0.2}}
\put(-9,-4){\circle*{0.2}}
\put(-9,-5){\circle*{0.2}}
\put(-10,-4){\circle*{0.2}}
\put(-11,-5){\circle*{0.2}}
\put(-9,-3){\line(0,-1){3}}
\put(-9,-3){\line(-1,-1){2}}
\put(0,1){\circle*{0.2}}
\put(0,.5){$\kappa$}
\put(-.3,0){$\scriptstyle{(0,1)}$}
  \put(0,2){\circle*{0.2}}
  \put(0,3){\circle*{0.2}}
   \put(0,4){\circle*{0.2}}
   \put(4,4){\circle*{0.2}}
   \put(4,3){\circle*{0.2}}
   \put(4,2){\circle*{0.2}} 
   \put(3.5,1.5){$\alpha$}
   \put(3.5,1){$\scriptstyle{(4,2)}$}
 \put(5,3){\circle*{0.2}}
  \put(6,4){\circle*{0.2}}
   \put(0,1){\line(0,1){4}}
   \put(4,2){\line(1,1){2}}
   \put(4,2){\line(0,1){3}}
   \put(-3,-1){\circle*{0.2}}
   \put (-2, 0){\circle*{0.2}}
   \put(-3,.3){$\scriptstyle{(-2,0)}$}
      \put(-3,-1){\line(1,1){1}}
\end{picture}
\end{figure}

The following  is a consequence of Proposition \ref{prop:Q_0-acyclic}, an 
algebraic version of the existence of a $v_1$-self map.

\begin{lem}
 \label{lem:Q0_periodic}
 Let $M$ be a bounded-below $\aone$-module which is $Q_0$-acyclic. Then 
 \begin{enumerate}
  \item 
  multiplication by $b$ on $\stext (\field, M )$ is an isomorphism, 
hence  $\stext (\field, M )$ is a $\e$-module;
  \item
  the stable isomorphism $M \otimes J \simeq \Sigma^{-6} \Omega^2 M$ induces an 
isomorphism 
  \[
   \stext^{s , t, \epsilon  } (\field, M)
  \cong  \stext^{s+2, t+6, \epsilon +1} (\field , M)
  \]
  of $\e$-modules, which corresponds to multiplication by $\alpha$.
 \end{enumerate}
In particular, if $M$ is reduced, then $\ext_{\aone}^{*,*} (\field, M)$ is 
determined by $\stext^{s \geq 0, t-s \geq 0 } (\field, M)$
via the periodicity isomorphism
\[
 \stext^{s , t  } (\field, M)
\cong 
 \stext^{s+4N , t +12N} (\field, M)
 \]
for $N \in \nat$ such that  $t-s  + 8N \geq 0$. 
\end{lem}

\begin{proof}
The first two statements follow from Proposition \ref{prop:Q_0-acyclic}, using 
the more precise formulation given in Theorem \ref{thm:yu} that  the stable 
isomorphism 
$\Sigma^{-6} \Omega^2 P_0 \rightarrow P_0 \otimes J$ is induced by $\alpha$ so 
that the stable isomorphism $\Sigma^{-12} \Omega^4 
P_0 \rightarrow P_0$ is induced by $b$.

The final statement follows by using periodicity with respect to multiplication 
by $b$ to extend to the upper half plane $s \geq 0$. The hypothesis that 
$M$ is reduced is required to identify $\stext^0$ with $\hom_\aone$ by Lemma 
\ref{lem:stable_hom_reduced}. 
\end{proof}

\begin{exam}
 Lemma \ref{lem:Q0_periodic} applies to:
 \begin{enumerate}
  \item 
  $M=P_0$; 
  \item
  the reduced cohomology of the truncated projective spaces, the modules 
$P^{2m}_{2n-1}$ of Section \ref{sect:trunc_p} (cf. 
\cite[Theorem 3.4]{davis});
  \item 
 the  modules $\wtwo (2n)$ of Section \ref{sect:brown-gitler}.
 \end{enumerate}
\end{exam}

\begin{exam}
\label{exam:Z_periodic}
 The case $M= Z$ in Lemma \ref{lem:Q0_periodic} is of particular interest.
There 
is a short exact sequence 
 of $\aone$-modules:
 \begin{eqnarray}
 \label{eqn:ses_Z}
    0
  \rightarrow 
  \field 
  \rightarrow 
  Z
  \rightarrow 
  \Sigma^{-1} \field 
  \rightarrow 
  0.
 \end{eqnarray}
 
Lemma \ref{lem:Q0_periodic} implies that $\stext (\field, Z) $ is $b$-periodic, 
hence is determined by the chart given 
in Figure \ref{fig:stext_Z}, which indicates that $\stext (\field, Z) $ is 
annihilated by $h_0$. 

\begin{figure}
 \caption{The chart for $\stext^{*,*} (\field,Z)$}
 \label{fig:stext_Z}
\begin{picture}(6,5)(-4,0)
\thinlines
\put(0,0){\circle*{0.2}}
\put(0,-.5){$\scriptstyle{(0,0)}$}
\put(1,1){\circle*{0.2}}
\put(2,2){\circle*{0.2}}
\put(2,1){\circle*{0.2}}
\put(2,.5){$\scriptstyle{(2,1)}$}
\put(3,2){\circle*{0.2}}
\put(4,3){\circle*{0.2}}
\put(0,0){\line(1,1){2}}
\put(2,1){\line(1,1){2}}
\end{picture}
\end{figure}

The long exact sequence for $\stext$ associated to the short exact sequence
(\ref{eqn:ses_Z}) has 
connecting morphism corresponding to multiplication by $h_0$. 
This leads  to an understanding of the symmetry for $\stext (\field, 
\field)$ as indicated in Figures \ref{fig:stext_F} and \ref{fig:stext_J}.
\end{exam}

A further useful fact is provided by the results of Section \ref{sect:basic}, 
which yield vanishing lines for $\stext$.

\begin{prop}
\label{prop:vanishing_lines}
 Let $M$ be a bounded-below, $Q_0$-acyclic $\aone$-module such that $H^* (M, 
Q_1)=0$ for $* \not \in [d_1, d_2]$, $d_1 \leq d_2 \leq \infty$.
 Then 
 $\stext^{s,t}(\field, M) =0$ for either of the following:
 \begin{eqnarray*}
  t-s &<&   2 s - d_2 -3 \\
  t-s  & > & 2 s - d_1.
 \end{eqnarray*}
 In particular, $h_0 ^{m} \stext (\field, M) = 0$, where $m = 1 + [\frac{d_2 - 
d_1 +3}{2}] $.
\end{prop}

\begin{proof}
 First consider the case $s =0$; we may assume that $M$ is reduced, so 
that 
the results of Section \ref{sect:basic} imply that 
 $
  \hom_{\aone} (\Sigma^k \field, M) = 0 
 $
  for $k < d_1$ and $k>d_2 + 3$, which gives the bounds in this case. 

 For the general case, by Lemma \ref{lem:Omega_Margolis}, the Margolis 
cohomology $H^* ((\Sigma^{-1} 
\Omega)^{-s} M, Q_1)$ is concentrated in $[d_1 - 2 s, d_2 - 
2s]$. The result now follows from the above.
\end{proof}

The groups $\stext (\field, 
A_{k,\epsilon})$ for $k \geq 1$ and $\epsilon \in \{0,1 \}$ can be calculated
using the 
presentation 
given  in Section \ref{sect:picard_interpret}:
\begin{eqnarray}
\label{eqn:AKe_ses}
 0
 \rightarrow
 A_{k,\epsilon}
 \rightarrow 
 (\Omega \Sigma^{-1}) ^{k+1} J^{\otimes \epsilon}
 \rightarrow 
 \field 
 \rightarrow
 0.
\end{eqnarray}

\begin{lem}
\label{lem:connecting_morphism}
 For $k \geq 1$ and $\epsilon \in \{0,1 \}$, the  morphism induced by $ (\Omega 
\Sigma^{-1}) ^{k+1} J^{\otimes \epsilon}
 \rightarrow 
 \field $,
 \[
  \stext (\field,  (\Omega \Sigma^{-1}) ^{k+1} J^{\otimes \epsilon} ) 
  \rightarrow 
  \stext (\field, \field),
 \]
is multiplication by $h_0^{k+1 - \epsilon } \kappa^{\epsilon}$. 
\end{lem}

\begin{proof}
 The element of $[\field, (\Omega \Sigma^{-1}) ^{-(k+1)} J^{\otimes 
\epsilon}]$ corresponding to the morphism is non-zero, hence corresponds to the 
unique non-zero class given.
\end{proof}

The short exact sequence (\ref{eqn:AKe_ses}) corresponds to a distinguished 
triangle in the stable module category:
\[
 \Omega \field 
 \rightarrow 
 A_{k,\epsilon}
 \rightarrow 
( \Omega \Sigma^{-1})^{k+1} J^{\otimes \epsilon }
\rightarrow .
\]

\begin{nota}
 For  $k \geq 1$ and $\epsilon \in \{0,1 \}$, denote by 
 \begin{enumerate}
  \item 
  $\mu _{k,\epsilon} \in \stext^{1,0} (\field , A_{k,\epsilon}) $ the image of 
$1 \in \stext^{0,0} (\field, \field)= \stext^{1,0}(\field, \Omega \field)$
under 
the morphism induced by $\Omega \field \rightarrow A_{k,\epsilon}$;
  \item 
  $\lambda_k \in \stext ^{k,k-3} (\field, A_{k,1})$ the class such that $b 
\lambda_k $ is the pre-image of the class $\alpha h_1 \in \stext (\field, J)$ 
under the morphism $ A_{k,1}
 \rightarrow  ( \Omega \Sigma^{-1})^{k+1} J$;
\item 
$\nu_k \in \stext^{k+2,k+3} (\field , A_{k,0})$ the pre-image of the class $h_1 
\in \stext (\field, \field)$ under the morphism $ A_{k,0}
 \rightarrow  ( \Omega \Sigma^{-1})^{k+1} \field$.
 \end{enumerate}
The Adams indexing of these classes are $|\mu_{k,\epsilon}|= (-1,1)$, 
$|\lambda_k |= (-3,k)$ and $|\nu_k|= (1,k+2)$. 
\end{nota}

\begin{rem}
 The classes $\lambda_k$ and $\nu_k$ are well-defined, since it is easily seen
that there are unique 
non-trivial classes in the given degrees (cf. Theorem 
\ref{thm:calculate_stext_Ake}). 
 The Adams indexing  can be determined as follows:
 \begin{enumerate}
  \item 
 $\alpha h_1$ corresponds to a class in $\stext 
(\field, ( \Omega \Sigma^{-1})^{k+1} J)$ of Adams indexing $(5,k+4)$  and the
class $b$ has Adams indexing 
$(8,4)$; 
  \item 
$h_1$ corresponds to a class in $\stext (\field, (\Omega \Sigma^{-1})^{k+1}
\field)$ of Adams indexing $(1,k+2)$.
 \end{enumerate} 
\end{rem}

The following result constitutes a complete calculation of the $\stext$-groups
for 
the modules which are classified by Theorem \ref{thm:classification} (up 
to shifting by elements of the Picard group $\pica$).

\begin{thm}
\label{thm:calculate_stext_Ake}
Let  $1 \leq k $ be an integer.
\begin{enumerate}
 \item 
There is an isomorphism of $\e$-modules:
\[
 \stext(\field , A_{k, 1} ) 
 \cong 
 \e \langle \mu_{k,1} , \lambda_k \rangle / \sim
\]
with relations:
\begin{eqnarray*}
a \lambda_k &=& \kappa \lambda_k  \phantom{h} = \phantom{h} h_0 \lambda_k  
\phantom{h} = \phantom{h} 0\\
h_1^2 \lambda_k &= & h_0^{k+1} \mu_{k,1} \\
h^k_0 \kappa & =&0, 
\end{eqnarray*}
hence $h^k_0 a \mu_{k,1} = 0$. 

In particular, $h_0^{k+2} \stext (\field , A_{k, 1} ) =0$.
\item 
There is an isomorphism of $\e$-modules:
\[
 \stext (\field , A_{k,0} ) 
 \cong 
 \e \langle \mu_{k,0} , \nu_k \rangle / \sim
\]
with relations:
\begin{eqnarray*}
a \nu_k &=& \kappa \nu_k  \phantom{h} = \phantom{h}  h_0 \nu_k  \phantom{h} = 
\phantom{h} 0\\
h_1^2 \nu_k &= & h_0^{k} a\mu_{k,0} \\
h^{k+1}_0 \mu_{k,0}& =&0. 
\end{eqnarray*}
In particular, $h^{k+1}_0  \stext (\field , A_{k,0} ) =0$.
\end{enumerate}
\end{thm}

\begin{proof}
In both cases, consider the long exact sequence for $\stext$ associated to the 
short exact sequence (\ref{eqn:AKe_ses}); Lemma \ref{lem:Q0_periodic} implies 
that $\alpha$ (hence $b$) is invertible on $\stext(\field, A_{k,\epsilon})$ 
hence the long exact sequence can be localized by inverting $\alpha$. The 
connecting morphism then corresponds to multiplication by $h_0^{k+1 -\epsilon} 
\kappa^{\epsilon}$ acting on $\e$, by Lemma \ref{lem:connecting_morphism}.
Since $\kappa^2=h_0 ^2$, in both case the kernel of multiplication by $h_0^{k+1 
-\epsilon} \kappa^{\epsilon}$ lies in the $h_0$-torsion of $\e$, which
coincides 
with the space annihilated by $h_0$. 
This torsion gives rise respectively to the classes generated over $\e$ by 
$\lambda_k$ and $\nu_k$.

From the long exact sequence, the  structure of $\stext(\field, A_{k,\epsilon})$
as a 
$\pica$-graded vector space follows. To identify the module structure, the only 
non-trivial points to verify  are the relations
\begin{enumerate}
 \item  
 $h_1^2 \lambda_k =  h_0^{k+1} \mu_{k,1}$;
 \item 
$ h_1^2 \nu_k = h_0^{k} a\mu_{k,0}$.
\end{enumerate}

These can be established by comparison with other known calculations, for 
example, via Proposition \ref{prop:identify_A} with  $\stext (\field, DP_0)$ 
(see 
\cite[Appendix A]{bg2}).
\end{proof}

\begin{exam}
 The  chart in Figure \ref{fig:A21} illustrates the case $A_{(2,1)}$ and Figure 
\ref{fig:A20}  the case $A_{2,0}$.
\end{exam}

\begin{figure}
\caption{The chart for $\stext(\field, A_{2,1})$}
\label{fig:A21}
\begin{picture}(5,7)(-1,-1)
\thinlines
 \put(0,2){\circle*{0.2}}
 \put(-.5, 1.5){$\lambda_2$}
  \put(1,3){\circle*{0.2}}
  \put(2,4){\circle*{0.2}}
 \put(2,3){\circle*{0.2}}
  \put(2,2){\circle*{0.2}}
  \put(2,1){\circle*{0.2}}
  \put(1.5,0.5){$\mu$}
   \put(1,0){$\scriptstyle{(-1,1)}$} 
   \put(3,2){\circle*{0.2}}
   \put(4,3){\circle*{0.2}}
   \put(6,4){\circle*{0.2}}
   \put(6,5){\circle*{0.2}}
   \put(5.5,3.5){$a\mu$}
      \put(5.5,3){$\scriptstyle{(3,4)}$} 
   \put(0,2){\line(1,1){2}}
   \put(2,1){\line(0,1){3}}
   \put(2,1){\line(1,1){2}}
   \put(6,4){\line(0,1){1}}
\end{picture}
\end{figure}

\begin{figure}
\caption{The chart for $\stext(\field, A_{2,0})$}\label{fig:A20}
\begin{picture}(5,7)(-1,-1)
\thinlines
 \put(0,0){\circle*{0.2}}
 \put(-.5,-.5){$\mu$}
    \put(-1,-1){$\scriptstyle{(-1,1)}$} 
  \put(0,1){\circle*{0.2}}
  \put(0,2){\circle*{0.2}}
 \put(1,1){\circle*{0.2}}
  \put(2,2){\circle*{0.2}}
  \put(2,3){\circle*{0.2}}
  \put(1.5,3.5){$\nu_2$}
   \put(3,4){\circle*{0.2}}
   \put(4,5){\circle*{0.2}}
   \put(4,4){\circle*{0.2}}
   \put(4,3){\circle*{0.2}} 
   \put(3.5,2.5){$a\mu$}
    \put(3.5,2){$\scriptstyle{(3,4)}$} 
   \put(0,0){\line(0,1){2}}
   \put(0,0){\line(1,1){2}}
   \put(2,3){\line(1,1){2}}
   \put(4,3){\line(0,1){2}}
\end{picture}
\end{figure}

\begin{rem}
\ 
\begin{enumerate}
 \item 
 Results corresponding to  Theorem \ref{thm:calculate_stext_Ake} occur in 
the work of Davis \cite[Theorem 3.4 and 
Lemma 3.10]{davis} for truncated projective spaces. The  usage of $\stext$ in 
place of $\ext$ greatly simplifies the presentation of 
these structures and the separation into the two cases $A_{k,0}$, $A_{k,1}$ 
makes the behaviour more transparent. 
 \item 
  By the analysis of the modules $\wtwo(2n)$ in Section 
\ref{sect:brown-gitler} 
and the stable isomorphism $\Sigma^{1+2^{\nu+1}} 
\wtwo(2^{\nu +1})\simeq A_{\nu,1}$, the case $\epsilon =1$ of Theorem 
\ref{thm:calculate_stext_Ake}  gives a calculation of 
$\stext (\field , \wtwo(2n))$, for any $n \in \nat$.
 \item 
Lemma \ref{lem:Q0_periodic} applies to give $\ext^{s,t}_{\aone} (\field , 
\wtwo(2n))$ for $s >0$, thus recovering the algebraic part of 
\cite[Theorem 1.1]{pearson}, by using the duality result Proposition 
\ref{prop:duality_A}. 
\end{enumerate}
\end{rem}

\begin{rem}
 There is more structure available when considering $\ext_{\aone} (\field , 
\wtwo(2n))$.
\begin{enumerate}
 \item
 The definition of $\wtwo(2n)$ via the weight splitting of 
$\wtwo:=  \field [\zeta_0^2, \zeta_1, \zeta_2 , \ldots ] $ 
 provides $\aone$-linear maps 
 \[
  \wtwo(2i) \otimes \wtwo (2j) 
  \rightarrow \wtwo(2i+2j) 
 \]
and hence $\stext (\field , \wtwo(2 *))$ has an algebra 
structure over $\e$. 
\item 
The spaces $\stext (\field , A_{k,1})$, for $k \in \nat$, have 
analogues of a Dieudonné module structure (cf. \cite{goerss}), where 
multiplication by $h_0$ on $\stext$ corresponds to multiplication by $2$ on 
stable homotopy groups. 
\end{enumerate}
\end{rem}

\section{Relating dual Brown-Gitler modules and $P_1$}
\label{sect:bgp}

The existence of a close relationship between the mod $2$ reduced cohomology of 
$\rp^\infty$ 
and the cohomology of the dual Brown-Gitler spectra $T(k)$ was exhibited 
algebraically by Miller \cite{miller} and  
Kuhn observed \cite{kuhn_new} that this has a topological counterpart, namely 
$\rp^\infty$ is a stable summand 
of 
\[
 \mathrm{hocolim} 
 \{ T (1) \rightarrow T (2) \rightarrow \ldots \rightarrow T (2^j) \rightarrow 
\ldots \},
\]
for any direct sequence of maps inducing the surjections $H^* (T(2^{j+1})) 
\twoheadrightarrow H^* (T (2^j))$ of the  Mahowald short exact sequences in
cohomology (cf. Lemma \ref{lem:mahowald_ses}).
 Recall that the reduced cohomology of $\rp^\infty$ is isomorphic to 
$\Sigma^{-1} 
\Omega P_0$ (see Section \ref{sect:applications}), which is denoted $P_1$ in 
\cite{bruner_Ossa}.

Theorem \ref{thm:identify_BG} yields the $\aone$ stable 
equivalence 
$
 H^* (T(2^{\nu +1}) )
 \simeq 
 \Sigma^{-1} A_{\nu,1},
$
for $\nu \geq 0$, and it is interesting to consider the analogue of the above 
in 
the stable module category.

\begin{lem}
\label{lem:exist_Anu1}
 For $\nu \in \nat$, there exists a morphism of $\aone$-modules 
 \[
  A_{\nu +1, 1} \rightarrow 
  A_{\nu, 1}
 \]
which induces an isomorphism on $H^3 ( -, Q_1)$.
\end{lem}

\begin{proof}
(Indications.)
 This can be checked by hand using the material of Section 
\ref{sect:picard_interpret}. (A better method is to calculate 
 $[A_{\nu +1, 1}, A_{\nu, 1}] \neq 0$; see Proposition 
\ref{prop:unicity_st_map_A} below.)
\end{proof}

\begin{rem}
\label{rem:indeterminacy}
 Some caution is in order since 
 $
  \hom _{\aone} (A_{2,1},A_{1,1}) = \field^{\oplus 2}
 $
and, in this case, there are two distinct morphisms satisfying the conclusion
of 
Lemma 
\ref{lem:exist_Anu1}. The complication arises from the presence of the Joker at 
the top of $A_{2,1}$, which allows for a non-trivial composite morphism:
\[
 A_{2,1} \twoheadrightarrow \Sigma^5 \field \hookrightarrow A_{1,1}
\]
which is clearly trivial on $H^3 (-,Q_1)$.
\end{rem}

\begin{lem}
\label{lem:exist_P1_A}
  For $1 \leq \nu \in \nat$, there exists a morphism of $\aone$-modules 
 \[
  \Omega P_0 \rightarrow 
  A_{\nu, 1}
 \]
which induces an isomorphism on $H^3 ( -, Q_1)$.
\end{lem}

\begin{proof}
 Recall from Section \ref{sect:killing} that there is a short exact sequence $0 
\rightarrow \Omega P_0 \rightarrow \Sigma R \rightarrow \field \rightarrow 0$  
of $\aone$-modules and, by construction (see Section 
\ref{sect:picard_interpret}),  
 there is a short exact sequence $0 \rightarrow A_{\nu, 1}\rightarrow 
(\Omega \Sigma^{-1})^{\nu +1} J \rightarrow \field \rightarrow 0$. For 
connectivity reasons, restricting a morphism $ \Sigma R 
\rightarrow  (\Omega \Sigma^{-1})^{\nu +1} J$ to elements of strictly 
positive degree induces a morphism $\Omega P_0 \rightarrow A_{\nu, 1}$. 
 
 Tensoring the projection $\Sigma R \twoheadrightarrow \field$ with 
$(\Omega \Sigma^{-1})^{\nu +1} J $ induces a stable morphism $\Sigma R 
\rightarrow (\Omega \Sigma^{-1})^{\nu +1} J$, since $(\Omega \Sigma^{-1})^{\nu 
+1} J \otimes \Sigma R \simeq \Sigma R$, which induces an isomorphism 
on $H^0 (-,Q_0)$. The restriction of a representative morphism $\Sigma R 
\rightarrow (\Omega \Sigma^{-1})^{\nu +1}  J$ of $\aone$-modules to strictly 
positive degree provides the required morphism. 
\end{proof}

This Lemma can be made more precise by the following:

\begin{prop}
\label{prop:inverse_system_Ake}
An inverse system of morphisms of $\aone$-modules 
 \[
  \{ f_\nu : A_{\nu +1, 1} \rightarrow A_{\nu,1} ~|~ \nu \geq 1 \}
 \]
such that $H^3 (f_{\nu}, Q_1)$ is an isomorphism, yields an isomorphism
\[
 \lim_{\leftarrow} A_{\nu, 1} \cong \Omega P_0.
\]
\end{prop}

\begin{proof}
(Indications.)
In a fixed degree, the inverse system stabilizes to the given isomorphism.
\end{proof}

The indeterminacy evoked in Remark \ref{rem:indeterminacy} is removed upon 
passage to the stable module category:

\begin{prop}
\label{prop:unicity_st_map_A}
 For $1 \leq k, l \in \nat$, 
 \[
  [A_{k,1}, A_{l,1}] = \left\{
  \begin{array}{ll}
   \field & k \geq l \\
   0 & k <l
  \end{array}
\right.
 \]
and the non-trivial morphism is detected on $H^3 (-,Q_1)$.
\end{prop}

\begin{proof}
 Consider the long exact sequence associated to $$0 \rightarrow A_{k, 
1}\rightarrow (\Omega \Sigma^{-1})^{k+1} J \rightarrow \field \rightarrow 
0,$$ of which the relevant portion is
 \[
  [(\Omega\Sigma^{-1})^{k+1}J, A_{l,1}]
  \rightarrow 
  [A_{k,1}, A_{l,1}]
  \rightarrow
  [\Omega \field,  A_{l,1}]
  \stackrel{h^k_0\kappa}{\rightarrow}
  [\Omega (\Omega \Sigma^{-1})^{k+1}J, A_{l,1}],
 \]
where the indicated map is given by Lemma \ref{lem:connecting_morphism}. 
 Theorem \ref{thm:calculate_stext_Ake}  shows that 
$[(\Omega \Sigma^{-1})^{k+1}J, A_{l,1}]=0$ and $[\Omega \field,  A_{l,1}] = 
\field$, generated by the class $\mu_{k,1}$. 
By Theorem  \ref{thm:calculate_stext_Ake}, $h^k_0 \kappa \mu_{k,1} \neq 0$ 
if and only if $k <l$; the result follows.
\end{proof}

Hence there is a well defined inverse system 
\[
 \ldots \rightarrow A_{\nu +1, 1} \rightarrow A_{\nu , 1} \rightarrow \ldots 
\rightarrow A_{1, 1}
\]
in the stable module category, where each morphism induces an isomorphism on 
$H^3 (-, Q_1)$. This induces an isomorphism 
\[
 \stext (\field, \Omega P_0) \cong \lim _\leftarrow  \stext (\field , A_{\nu , 
1} ) .
\]

\begin{rem}
Topologically this corresponds to the fact that the stable summand of
$
 \mathrm{hocolim} 
 \{ T (1) \rightarrow T (2) \rightarrow \ldots \rightarrow T (2^j) \rightarrow 
\ldots \} 
$ 
complementary to $\Sigma^{\infty} \rp^\infty$ is $KO^*$-acyclic. 
\end{rem}

\begin{rem}
A similar result holds for a positive integer $n$, when considering 
\[
 T (n) \rightarrow T (2n) \rightarrow \ldots \rightarrow T (2^j n ) \rightarrow 
\ldots ,
\]
by Theorem \ref{thm:identify_BG}, which provides a stable equivalence (for $n$ 
odd and $j \geq 1$):
\[
 \Sigma^{2^jn} \wtwo(2^jn)
\simeq (\Sigma \Omega^{-1} ) ^{1- \alpha (n)} \big( \Sigma^{-1} A_{ j-1, 1} 
\big).
\]
In particular, the term $(\Sigma \Omega^{-1}) ^{1-\alpha (n)}$ is 
independent of $j$, hence Proposition \ref{prop:unicity_st_map_A} gives
the analogous $\stext$ calculation.
\end{rem}

\section{Stable decompositions of tensor products}
\label{sect:tensor}

It is interesting to consider the stable isomorphism type of tensor products of 
modules of the form appearing in Theorem \ref{thm:classification}. This 
problem has been addressed for truncated projective spaces by Davis in 
\cite[Theorem 3.9]{davis}.

Over  $\eone$, the calculation is straightforward:
the tensor products always split as a direct sum of two, non-trivial 
indecomposable modules. Over $\aone$ the situation is more delicate. For 
example:

\begin{exam}
\label{exam:Z_tensor_Z}
 Recall that $Z$ is the two-dimensional $\aone$-module in degrees $-1,0$, with 
$Sq^1$ acting 
non-trivially. The tensor product $Z \otimes Z$ is indecomposable, with 
structure:
 
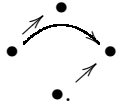
\[
  \xymatrix@R=.5pc @C=.5pc{
  &
  \bullet 
  \\
  \bullet 
  \ar@<.5ex>[ur]
  \ar@<-.5ex>@/^1pc/[rr]
  &&
  \bullet 
  \\
  &
  \bullet.
  \ar[ru]
  }
 \]
 Over $\eone$, the $Sq^2$ is not seen, and the module splits as $(\Sigma^{-1} Z 
\oplus Z)|_{\eone}$. Since $Z$ generates (under $\Sigma^{-3}\Omega$) all stable 
isomorphism classes of $\aone$-modules corresponding to $s=1$ in Theorem 
\ref{thm:classification}, a similar statement holds for tensor products of any 
two such modules.
\end{exam}

Recall from Section \ref{sect:picard_interpret} (in particular Corollary 
\ref{cor:orbit}) that, for $1 \leq k$,  the modules 
$A_{k,0}$, $A_{k,1}$ give orbit representatives for  modules of the form 
occurring in Theorem 
\ref{thm:classification}. To consider the stable isomorphism of tensor products 
of such modules, it suffices to consider 
$ 
 A_{k,\epsilon} \otimes A_{l, \delta}$ and, without loss of generality, we may 
suppose that $k \leq l$. 

The module $A_{l, \delta}$ is defined by a short exact sequence 
\[
 0
 \rightarrow 
 A_{l, \delta} 
 \rightarrow 
 (\Omega \Sigma^{-1}) ^{l+1} J ^{\otimes \delta} 
 \rightarrow 
 \field 
 \rightarrow 
 0
\]
and the analogous statement holds for $A_{k,\epsilon}$. Forming the tensor 
product $A_{k, \epsilon} \otimes - $ with the above short exact sequence gives 
\[0
 \rightarrow 
 A_{k, \epsilon} \otimes A_{l, \delta} 
 \rightarrow 
 (\Omega \Sigma^{-1}) ^{l+1} J ^{\otimes \delta} \otimes A_{k, \epsilon}
 \rightarrow 
 A_{k,\epsilon}
 \rightarrow 
 0
\]
and hence a distinguished triangle in the stable module category:
\begin{eqnarray}
 \label{eqn:dist_triang}
 \Omega A_{k,\epsilon}
 \rightarrow
 A_{k, \epsilon} \otimes A_{l, \delta} 
 \rightarrow 
 (\Omega \Sigma^{-1}) ^{l+1} J ^{\otimes \delta} \otimes A_{k, \epsilon}
 \rightarrow
\end{eqnarray}
with connecting  morphism $ (\Omega \Sigma^{-1}) ^{l+1} J ^{\otimes \delta} 
\otimes A_{k, \epsilon}
 \rightarrow  A_{k,\epsilon}$. This distinguished triangle splits to provide a 
direct sum decomposition of $A_{k, \epsilon} \otimes A_{l, \delta}$ if and only 
if 
 the corresponding element of 
 \[
 \obstruct_{k,l,\delta,\epsilon} \in 
 [ (\Omega \Sigma^{-1}) ^{l+1} J ^{\otimes \delta} \otimes A_{k, \epsilon}, 
A_{k,\epsilon}]
\]
is trivial.

The defining short exact sequence for $A_{k,\epsilon}$ gives an exact sequence:
\[
[ (\Omega \Sigma^{-1}) ^{k+ l+2} J^{\otimes (\delta + \epsilon)},
A_{k,\epsilon}]
 \rightarrow 
[(\Omega \Sigma^{-1}) ^{l+1} J ^{\otimes \delta} \otimes A_{k, \epsilon}, 
A_{k,\epsilon}]
 \rightarrow 
 [ \Omega (\Omega \Sigma^{-1}) ^{l+1} J^{\otimes \delta}, A_{k,\epsilon}].
\]

\begin{lem}
\label{lem:stext_first_term}
 For $1 \leq k\leq l \in \nat$ and $\delta,\epsilon \in \{0,1 \}$, 
 \[
  [ (\Omega \Sigma^{-1}) ^{k+ l+2} J^{\otimes (\delta + \epsilon)},
A_{k,\epsilon}]
  = 
0
\]
\end{lem}

\begin{proof}
This follows from the stable ext calculations of Theorem
\ref{thm:calculate_stext_Ake}.

If $\delta + \epsilon \equiv 0 \mod (2)$, then the relevant  group is 
$\stext (\field, A_{k,\epsilon})$ of Adams indexing $t-s=0$, $s=k+l+2$; 
for $t-s=0$,  $\stext(\field ,A_{k, \epsilon}) $ is $\field$ concentrated in 
degree $s =2$. The hypothesis on $k,l$ implies that $k+l+2 >2$, whence the
result in this case.
 
In the remaining case, Proposition \ref{prop:Q_0-acyclic} implies that there is
a stable isomorphism $J \otimes A_{k, \epsilon} 
\simeq \Sigma^4 (\Omega^{-1} \Sigma)^2 A_{k, \epsilon}$, hence 
\[
 [ (\Omega \Sigma^{-1}) ^{k+ l+2} J, A_{k,\epsilon}]
\cong
 [  (\Omega \Sigma^{-1}) ^{k+ l+4}\field, \Sigma ^4 A_{k,\epsilon}]
\]
and the latter group is trivial. 
\end{proof}

It follows that $\obstruct_{k,l,\delta,\epsilon} \in 
 [ (\Omega \Sigma^{-1}) ^{l+1} J ^{\otimes \delta} \otimes A_{k, \epsilon}, 
A_{k,\epsilon}]$ is zero if and only if its image 
\[
 \overline{\obstruct_{k,l,\delta,\epsilon}} \in 
 [ \Omega (\Omega \Sigma^{-1}) ^{l+1} J^{\otimes \delta}, A_{k,\epsilon}]
\]
is zero. By construction, $\overline{\obstruct_{k,l,\delta,\epsilon}}$ is the
composite of the commutative diagram in the stable module category:
\[
 \xymatrix{
 (\Omega \Sigma^{-1}) ^{l+1} J^{\otimes \delta} \otimes \Omega \field
 \ar[rr]
 \ar[d]
 &&
 \Omega \field \ar[d]^{\mu_{k,\epsilon}}
 \\
  (\Omega \Sigma^{-1}) ^{l+1} J ^{\otimes \delta} \otimes A_{k, \epsilon}
  \ar[rr]_(.6){\obstruct_{k,l,\delta,\epsilon}}
  &&
  A_{k,\epsilon},
 }
\]
where the vertical morphisms are induced by $\mu_{k,\epsilon} : \Omega \field
\rightarrow  A_{k,\epsilon}$ and the horizontal morphisms by the non-trivial
morphism $ (\Omega \Sigma^{-1}) ^{l+1} J ^{\otimes \delta} \rightarrow \field$.
This provides the identification:
\[
 \overline{\obstruct_{k,l,\delta,\epsilon}} = h_0 ^{l+1-\delta} \kappa^\delta
\mu_{k,\epsilon}.
\]

\begin{thm}
  \label{thm:tensor_decomp}
  For  $1 \leq k\leq l \in \nat$ and $\delta,\epsilon \in \{0,1 \}$, the 
distinguished 
triangle (\ref{eqn:dist_triang}) splits to give a stable isomorphism 
  \[
   A_{k, \epsilon} \otimes A_{l, \delta} 
   \simeq 
   \Omega A_{k,\epsilon}
\oplus
   (\Omega \Sigma^{-1}) ^{l+1} J ^{\otimes \delta} \otimes A_{k, \epsilon}
  \]
if and only if $\{k,l ,\delta, \epsilon \}$ does not satisfy  both the following
conditions 
\begin{eqnarray*}
 k&=& l\\
\delta + \epsilon &\equiv& 1 \mod (2).
\end{eqnarray*}
\end{thm}

\begin{proof}
 By the previous discussion, the distinguished triangle splits if and only if
$\overline{\obstruct_{k,l,\delta,\epsilon}} = 0$,  which is equivalent to 
 the condition $ h_0 ^{l+1-\delta} \kappa^\delta \mu_{k,\epsilon} =0$. Note that
$1 \leq k \leq l$ by hypothesis. 
 \begin{enumerate}
 \item 
 In the case $\delta=0$, by Theorem \ref{thm:calculate_stext_Ake},  $ h_0 ^{l+1}
 \mu_{k,\epsilon}\neq 0 $ if and only if $\epsilon =1$ and $k=l$.
 \item 
 If $\delta=1$, since $\alpha$ is invertible, one can equivalently consider
$\alpha h_0 ^{l} \kappa \mu_{k,\epsilon} = h_0 ^{l}a \mu_{k,\epsilon}$. By
Theorem \ref{thm:calculate_stext_Ake}, this is nonzero if and only if $k=l$ and
$\epsilon =0$. 
\end{enumerate}
Thus $\overline{\obstruct_{k,l,\delta,\epsilon}} \neq  0$ if and only if $k=l$
and $\delta + \epsilon \equiv 1 \mod (2)$, as required.
\end{proof}

 \begin{exam}
  Theorem \ref{thm:tensor_decomp} implies that $A_{k,1} \otimes A_{l,1}$ always 
splits (supposing $k,l\geq 1$). Since truncated projective spaces can lie in
either of the components, by 
Theorem \ref{thm:identify_A_trunc}, tensor 
products of truncated projective spaces need not split, as observed  in 
\cite[Theorem 3.9]{davis}.
 \end{exam}

\section{The Toda complex and the Postnikov tower of $ko$}
\label{sect:toda}

In \cite{powell}, the cohomology $ko \langle n \rangle 
^* 
(BV)$ of elementary abelian $2$-groups was studied as a functor of $V$, for $n 
\in \nat$. 
The emphasis there was on the functorial structure and the use of the known 
structure of the periodic theory $KO^* (BV)$ and the key result was the proof 
that the morphism of 
spectra derived from the Postnikov tower for $ko$
\[
 ko \langle n \rangle 
 \rightarrow 
 KO 
 \vee 
 \Sigma^n H(KO_n)
\]
(where $H \pi$ denotes the Eilenberg-MacLane spectrum for the abelian group 
$\pi$), induces an injection on the 
cohomology of $BV$.
In \cite{powell}, this was expressed in terms of the property of 
detection; from the viewpoint of the current paper, where  functoriality is a
less powerful tool, 
it is natural to reinterpret this  as a statement that the Adams spectral 
sequence for $ko\langle n \rangle ^* (BV)$ collapses at the $E_2$-term. The 
purpose of this section is to indicate the analogous results which hold for 
Brown-Gitler spectra.

The main result of \cite{pearson} studies the $ko$-homology of Brown-Gitler 
spectra and shows that the Adams spectral sequence collapses at $E_2$ in this 
case. This fact follows easily from Theorem \ref{thm:calculate_stext_Ake} here.
A similar statement holds for $ko$-cohomology; moreover, the conclusion extends 
to the theories $ko \langle n \rangle$, as in \cite{powell}. 

Recall (cf. \cite{powell} and \cite[Section A.5]{bg2}) that there is an exact 
Toda complex $\cale_\bullet$ of the form 
\[
 \ldots 
 \rightarrow 
 \Sigma ^n \calc _n 
 \rightarrow 
 \Sigma^{n-1} \calc_{n-1} 
 \rightarrow 
 \ldots 
\]
where the $\calc_i$ are four-periodic up to suspension, 
$
 \calc_{i+4} = \Sigma^8 \calc_i, 
$
the differentials are of degree zero  and 
\begin{eqnarray*}
 \calc_0 &=& \aone \otimes_\azero \field \\
 \calc_1 &=& \Sigma^1 \aone \\ 
\calc_2 &=& \Sigma^2 \aone \\
\calc_3 &=& \Sigma^4 \aone \otimes_\azero \field.
\end{eqnarray*}
(See \cite[Section 4]{powell} and \cite[Figure A.5.6]{bg2} for the 
identification of the differentials;  in  \cite[Section 4]{powell}, the 
differentials are taken of degree one, hence the suspension is omitted.) 

The exact complex is obtained by splicing short exact sequences:
\[
\xymatrix@R=.75pc @C=.75pc{
&
\Sigma^3 \calc_3 
\ar@{->>}[rd]
&&
\Sigma^2 \calc_2 
\ar@{->>}[rd]
&&
\Sigma \calc_1 
\ar@{->>}[rd]
&&
\calc _0 
\ar@{->>}[rd]
\\
\Sigma^4 (\Sigma^8 \field) 
\ar@{^(->}[ur]
&&
\Sigma^3 K_3 
\ar@{^(->}[ur]
&&
\Sigma^2 K_2 
\ar@{^(->}[ur]
&&
\Sigma K_1 
\ar@{^(->}[ur]
&&
K_0 = \field  
}
\]
where 
\begin{eqnarray*}
 K_1 &=& \Sigma^5 \Omega ^{-1} J \\
 K_2 &=& \Sigma^4 J \\
 K_3 &=& \Sigma^3 \Omega J.
\end{eqnarray*}
In particular, the $K_j$'s represent elements of the Picard group and they are 
cyclic $\aone$-modules. Moreover, for $n \in \nat$, there is an isomorphism of 
$\cala$-modules: 
\[
 H\field_2^* (ko \langle n \rangle) \cong \cala \otimes_{\aone} K_n.
\]
This makes it transparent that the results of the previous sections can be 
applied to the Adams spectral sequence calculations for $ko\langle n \rangle 
^*$. In particular, on stable ext groups one has:

\begin{lem}
\label{lem:stable_ext_Toda}
(Cf. \cite{powell}.) 
 For $M$ a $Q_0$-acyclic $\aone$-module and $n \in \nat$,  the connecting 
morphisms of the Toda exact sequence induce a natural isomorphism:
 \[
  \stext^{s,t} (K_n , M) 
  \cong 
  \stext^{s-n , t-n} (\field, M).
 \]
\end{lem}

\begin{proof}
Straightforward: this follows from relative homological 
algebra for the pair $(\aone, \azero)$ (compare Lemma \ref{lem:Q0_periodic}.)
\end{proof}

\begin{rem}
In the case of Adams spectral sequence calculations for $ko \langle n \rangle^* 
(X) $, the $ko \langle n \rangle$-cohomology of a spectrum $X$ with 
$Q_0$-acyclic mod $2$-cohomology, this corresponds to the relationship with the 
periodic cohomology $KO^* (X)$. 
For example, it applies in the case of the suspension spectra $\Sigma^\infty 
BV$ 
and to the $2$-torsion Brown-Gitler spectra and their 
Spanier-Whitehead duals.
\end{rem}

In the study of the detection property introduced in \cite{powell}, the  
homology of the complex $\hom (\cale_\bullet, M)$ plays a fundamental role. In 
relation with Lemma \ref{lem:stable_ext_Toda}, one has the following 
identification:

\begin{lem}
 \label{lem:truncate_Toda_ext}
 For $M$ a $Q_0$-acyclic $\aone$-module and $1 \leq n \in \zed$, there is a 
natural isomorphism:
 \[
  H^n (\hom_{\aone} (\cale_\bullet, M) ) 
  \cong 
  \ext^n _\aone (\field, M). 
 \]
\end{lem}

\begin{proof}
 Similar to Lemma \ref{lem:stable_ext_Toda}. 
\end{proof}

The exclusion of the case $n=0$ above is due to the difference between 
morphisms 
in the category of $\aone$-modules and in the stable module category. 

Putting these results together, one obtains the analogue of the main result of 
\cite{powell} for the cohomology of Brown-Gitler spectra with respect to the 
theories $ko \langle n \rangle$.

Recall from Section \ref{sect:brown-gitler} that, for $l \in \nat$,  $B(l)$ 
denotes the $l$th Brown-Gitler spectrum 
 and $DB(l)$ its Spanier-Whitehead dual.

\begin{thm}
 \label{thm:detection_BG}
For $0< l \in \nat$ and $n \in \nat$, the canonical morphisms of spectra $ 
ko\langle n \rangle \rightarrow KO$ and $ ko\langle n \rangle \rightarrow 
\Sigma^n H(KO_n)$ induce an injection
\[
 ko\langle n \rangle^* ( D B (l)) 
 \hookrightarrow KO^* (DB(l)) 
 \oplus 
 H(KO_n)^{*+n} (DB(l)).
\]
 \end{thm}
 
 \begin{proof}
 (Indications.) 
  The essential step is the proof that the Adams spectral sequence for 
$ko\langle n \rangle^*( D B (l)) $ collapses at $E_2$ and hence for $KO^*$, 
which follows as in \cite{davis}, for example. The 
result is then a consequence of the general results on detection established in 
\cite{powell}, 
  using Lemmas \ref{lem:stable_ext_Toda} and \ref{lem:truncate_Toda_ext} above. 
 \end{proof}

 \begin{rem}
  This result may be restated  as the injectivity of 
\[
ko\langle n \rangle_* (B (l)) 
 \hookrightarrow KO_* (B(l)) 
 \oplus 
 H(KO_n)_{*-n} (B(l)).
\]
As explained in \cite[Section 11]{goerss}, for $E$ a ring spectrum with 
representing spaces $\{\underline{E}_\bullet \}$ of the associated 
$\Omega$-spectrum, 
there is a surjective homomorphism 
\[
 E_* (B(l)) \twoheadrightarrow D_l H_* (\underline{E}_{l-*}), 
\]
that is compatible with the respective Frobenius and Verschiebung morphisms by 
\cite[Proposition 11.3]{goerss}. Here, for $H$ a graded, bicommutative 
Hopf algebra, $D_*H$ denotes the associated Dieudonné module; under the 
appropriate hypotheses, $D_* H$ determines $H$. 

It follows that Theorem \ref{thm:detection_BG} provides information on 
$H_* (\underline{ko\langle n \rangle }_\bullet)$, in particular  shedding light on 
results of Cowen Morton \cite{cowen_morton}.
\end{rem}



\providecommand{\bysame}{\leavevmode\hbox to3em{\hrulefill}\thinspace}
\providecommand{\MR}{\relax\ifhmode\unskip\space\fi MR }
\providecommand{\MRhref}[2]{%
  \href{http://www.ams.org/mathscinet-getitem?mr=#1}{#2}
}
\providecommand{\href}[2]{#2}

\end{document}